\documentclass[a4paper]{article}

\usepackage{amsmath}
\usepackage{amsthm}
\usepackage{amssymb}
\usepackage{enumerate}
\usepackage{booktabs}
\usepackage{rotating}
\usepackage{verbatim}
\usepackage[numbers]{natbib}
\usepackage{subcaption}
\usepackage{graphicx}
\usepackage{hyphenat}
\usepackage{hyperref}
\usepackage{appendix}
\usepackage{units}
\usepackage[affil-it]{authblk}
\usepackage{epstopdf}
\usepackage{adjustbox}
\usepackage{pgfplots}
\usepackage{float}
\pgfplotsset{compat=1.11}

\allowdisplaybreaks

\title{Algorithms for Euclidean Degree Bounded Spanning Tree Problems}

\author{Patrick J. Andersen%
  \thanks{Email: \texttt{pat.j.andersen@gmail.com}; Phone: \texttt{+61-413579238}; Corresponding author}
}

\author{Charl J. Ras}

\affil{School of Mathematics and Statistics, The University of Melbourne, Australia}

\newtheorem{lemma}{Lemma}

\providecommand{\keywords}[1]{\textit{Keywords: } #1}

\floatstyle{ruled}
\newfloat{algorithm}{htb}{alg}
\floatname{algorithm}{Algorithm}

\makeatletter
\def\@maketitle{%
  \newpage
  \null
  \vskip 2em%
  \begin{center}%
  \let \footnote \thanks
    {\Large\bfseries \@title \par}%
    \vskip 1.5em%
    {\normalsize
      \lineskip .5em%
      \begin{tabular}[t]{c}%
        \@author
      \end{tabular}\par}%
    \vskip 1em%
    {\normalsize \@date}%
  \end{center}%
  \par
  \vskip 1.5em}
\makeatother

\begin{document}
\maketitle
\begin{abstract}
Given a set of points in the Euclidean plane, the Euclidean \textit{$\delta$-minimum spanning tree} ($\delta$-MST) problem is the problem of finding a spanning tree with maximum degree no more than $\delta$ for the set of points such the sum of the total length of its edges is minimum. Similarly, the Euclidean \textit{$\delta$-minimum bottleneck spanning tree} ($\delta$-MBST) problem, is the problem of finding a degree-bounded spanning tree for a set of points in the plane such that the length of the longest edge is minimum. When $\delta \leq 4$, these two problems may yield disjoint sets of optimal solutions for the same set of points. In this paper, we perform computational experiments to compare the accuracies of a variety of heuristic and approximation algorithms for both these problems. We develop heuristics for these problems and compare them with existing algorithms. We also describe a new type of edge swap algorithm for these problems that outperforms all the algorithms we tested. \\\\
\end{abstract}
\section{Introduction}
Spanning tree problems are some of the most well-studied problems of combinatorial optimisation. They arise in a multitude of practical settings including computer and telecommunication networks, transportation, plumbing, and electrical circuit design \cite{graham1985history}. In graph theoretic terms, a typical spanning tree problem involves finding a subgraph $G'$ of a given graph $G = (V,E)$, such that the vertex set of $G'$ is $V$, $G'$ is connected, and $G'$ has $|V| - 1$ edges. We will focus on the Euclidean versions of spanning tree problems, where the nodes of our network correspond to a set of points $P$ in the Euclidean plane such that our input graph is the complete graph for $P$. In this way, we can assign lengths to each of the edges in the graph, where the length of the edge is given by the Euclidean distance between its endpoints.\\\\
In this paper, we will consider two separate types of spanning tree problems; the classic minimum spanning tree, and its bottleneck version. In the \textit{minimum spanning tree} (MST) problem, we require that the sum of lengths of the edges between nodes of our spanning tree is minimum. In the \textit{bottleneck spanning tree problem}, we require that the length of the longest edge in the network is minimum. Both of these spanning tree problems can be solved efficiently using polynomial time algorithms (see Kruskal \cite{kruskal1956shortest}, Prim \cite{prim1957shortest} for the MST problem, and Camerini \cite{mbst} for the MBST problem), and in fact every algorithm for the MST problem can be used to solve the MBST problem since every MST is a MBST.\\\\
For particular applications of spanning trees, additional constraints are required for a feasible spanning tree. The constraint we explore in this paper is the degree bound constraint, which is a requirement that the degrees of all nodes in the spanning tree be at most some constant. We denote this constant by $\delta$ and we refer to the constrained versions of the MST and MBST problems as the $\delta$-MST problem and $\delta$-MBST problem respectively. For the Euclidean versions, it is known that no vertex in an MST has degree greater than 6 and that there always exists an MST in which no vertex has degree greater than 5 \cite{degree5EMST}. Hence when $\delta \geq 5$, a solution to the MST problem for a given point set will also be a solution to both the $\delta$-MST and $\delta$-MBST problems. As such, we will only consider degree constraints in which $2 \leq \delta \leq 4$. It was shown by Papadimitriou and Vazirani \cite{papavazi} that the $\delta$-MST problem is NP-hard for $\delta = 2$ and $3$, and it was later shown by Francke and Hoffman \cite{francke2009Euclidean} to also be NP-hard for $\delta = 4$. For the $\delta$-MBST problem, it has been shown that the problem is NP-hard for $\delta = 2$ and $3$ \cite{andersen2016minimum}, but determining complexity of the case where $\delta = 4$ remains an open problem. Approximation algorithms for the Euclidean $\delta$-MST problem have been developed by Khuller et. al. \cite{khuller1996low} and Chan \cite{chan2003euclidean} and their approximation ratios for the $\delta$-MBST problem were explored in \cite{andersen2016minimum}. Heuristics for the general version of the $\delta$-MST problem have been explored by numerous authors including Narula and Ho \cite{narula1980degree}, Knowles and Corne \cite{knowles2000new}, and Bui and Zrncic \cite{bui2006ant}.\\\\
For the special case of $\delta = 2$, the $2$-MST problem is equivalent to the path version of the famous travelling salesman problem (TSP). The TSP has been well studied throughout the history of combinatorial optimization, and there are many algorithms for the problem that have been proposed and tested throughout the literature. There are numerous approximation and heuristic algorithms that have been proposed for the metric and Euclidean versions of the problem, with the more famous ones including Christofides algorithm \cite{christofides1976worst}, the Lin-Kernighan algorithm \cite{lin1973effective}, and Arora's PTAS \cite{arora1996polynomial}, among others \cite{rosenkrantz1977analysis},\cite{laporte1992traveling}. Whilst the ever-expanding literature on the TSP is incredibly vast, we will restrict our attention to algorithms that can be effectively applied to the path version of the TSP (TSP-path problem). The TSP-path problem has received significantly less attention than the usual cycle version, although there still various algorithms in the literature \cite{hoogeveen1991analysis}, \cite{an2015improving}, \cite{traub2018approaching}. On the other hand, the $2$-MBST problem is equivalent to the TSP-path problem with the bottleneck objective. An efficient 2-factor approximation for the metric TSP with bottleneck objective has been given by Parker and Rardin \cite{parker1984guaranteed}, however there has been little work done on specific algorithms for the Euclidean bottleneck TSP-path problem.\\\\
In this paper we survey existing algorithms and introduce a number of new algorithms for both the $\delta$-MST and $\delta$-MBST problems. We then perform computational testing of heuristics and approximation algorithms for the Euclidean $\delta$-MST and $\delta$-MBST problems in order to test their performances with respect to both the total weight objective criterion of the $\delta$-MST and the bottleneck length objective of the $\delta$-MBST problem. Our goal for this investigation was to answer the following three research questions. 
\begin{enumerate}
\item Should different algorithmic approaches be employed for the different degree bounds when it comes to solving the  Euclidean $\delta$-MST and Euclidean $\delta$-MBST problems, or is it preferable to use a single algorithm or its generalisation for all $\delta$?
\item Should different algorithmic approaches be employed for the Euclidean $\delta$-MST problem as opposed to the Euclidean $\delta$-MBST, or is it preferable to use a single algorithm for both objective criteria?
\item For any particular $\delta$, are there any algorithms that stand out as being suitable candidates for efficiently solving either the Euclidean $\delta$-MST or $\delta$-MBST problems?
\end{enumerate}

In regards to the to third question, our focus is less on the finding the ``best" algorithm for a specific problem variant, and more on identifying an algorithmic framework that seems to fare better at exploiting the structure of the problem and its instances. Our hope is that after seeing which frameworks are more effective, in the future we will be able to develop more sophisticated algorithms based upon the more successful frameworks. As such, we will only implement a select number of interesting existing algorithms within the vast array of approximation and heuristic algorithms available in the literature. Furthermore, we will favour approaches that are designed specifically for the problems we are analysing, such as approximation algorithms, over more general meta-heuristics.\\\\
\textbf{Our results:}
In order to have a sufficient set of algorithms for comparison, we describe a number of existing algorithms for the $\delta$-MST and $\delta$-MBST problems, and we develop a set of new edge swap algorithms for solving the $\delta$-MST and $\delta$-MBST problems based upon the ideas of local search. These new algorithms serve as benchmark comparisons for the other algorithms and as potential proof-of-concept algorithms. Since degree four and five nodes are relatively uncommon in MST's on uniformly distributed random point sets, we also develop an algorithm for producing test instances in which the MST's are guaranteed to have nodes of high degree.\\
After performing our computational experiments and comparing the results of the various algorithms, our answers to the research questions are as follows.
\begin{enumerate}
\item There are algorithms that outperform all others for specific $\delta$ but not for other degree bounds.
\item For certain degree bounds there are algorithms that outperform all others for the $\delta$-MST problem, but not for the $\delta$-MBST problem. This occurred particularly for the case when $\delta = 2$.
\item For the $2$-MBST, the Cube2 algorithm (Section~\ref{sec:cube2}) clearly outperformed all others tested, the $\delta$-Prim's algorithm (Section~\ref{sec:dPrim}) outperformed the other tested algorithms for the $3$-MBST problem, the Chan4 algorithm (Section~\ref{sec:Chan4}) outperformed the other tested algorithms for the $4$-MBST problem, and our newly  created DNLS algorithm (Section~\ref{sec:DNLS}) outperformed all the others tested for the $3$-MST and $4$-MST problems. Although the DNLS algorithm's time complexity may disqualify it as a suitable candidate for practical use, it does demonstrate a successful proof of concept for this kind of approach and it may be modified in future to make it more efficient.
\end{enumerate}


\section{Test Instances}
Throughout this paper, we denote the degree bound of our problems by $\delta$, where $\delta$ is always assumed to be 2, 3 or 4, and we let $n$ denote the number of input points in an instance. In order to test the performance of our algorithms with some degree of accuracy, we must have a sufficient number of instances upon which our algorithm can be tested. These test instances are lists of points in the Euclidean plane, where the number of points in a list is a pre-determined parameter $n$. A simple way of generating such instances would be to choose a random set of $n$ points in the plane using a bounded uniform distribution. Such an approach is suitable for the problems with $\delta = 2$, however complications arise when considering degree bounds of three or four. This is due to the very low frequency of degree four and five vertices that naturally occur in MST's for random point sets in the plane. Therefore when $\delta = 4$ or $5$, unless $n$ is very large, it is likely that an MST for a random point set in the plane will be a $\delta$-MST. Steele, Shepp and Eddy \cite{steele1987number} show that if the distribution of the points has no singular part, then as $n \rightarrow \infty$, the expected proportion of nodes of a given degree $D$ in the MST for the point set approaches some constant $\beta (D)$. In addition, the authors show experimentally that for uniformly distributed points in the plane with large $n$, on average approximately $0.7 \%$ of the vertices in MST for the point set will be of degree four, whereas vertices of degree five are even more rare and do not usually appear in instances. Thus, since we want to avoid most instances in which an MST is a $\delta$-MST, we opt for crafted instances in which we can be guaranteed certain numbers of degree four or five vertices.\\
Our approach for forcing vertices of degree $D =4$ or $5$ in the MST is to generate a set of $D+1$ point in the planes, which we refer to as \textit{star points}, such that the MST for a set of star points is a star graph $S_D$, i.e., a tree with a single vertex of degree $D$ and $D$ leaves. As long as all other generated points lie outside a certain region with respect to a given set of star points, the MST for the full point set will contain the star $S_D$ as an induced subgraph of the given star points. Hence the MST for the full point set will contain a vertex of degree $D$. In the following sections, we describe our chosen process for generating star points in a way that allows our crafted instances to be as random as possible whilst still ensuring vertices of specified degree.

\subsection{Star Points}
In this section, we give the necessary conditions for a star on a point set $S$ in the plane to be an MST for $S$, where either $|S| = 5$ to give the star $S_4$, or $|S| = 6$ to give the star $S_5$. For the $S_4$ case, we let $G = (\{v,v_1,v_2,v_3,v_4\},$ $\{(v,v_1),(v,v_2),(v,v_3),(v,v_4)\})$ be a star on a set of five points in the plane, where $v$ is the vertex of degree four, and $v_1, v_2, v_3, v_4$ are the leaf vertices in clockwise order. Let  $\theta_1,\theta_2,\theta_3,\theta_4$ be the angles between the edges incident to $v$ as shown in Figure~\ref{fig:star4fig}. The notation for the $S_5$ case is similar and is given by Figure~\ref{fig:star5fig}. We refer to the point $v$ as the \textit{centre point} and the other points as \textit{radial points}.

\begin{figure}[ht!]  
  \begin{subfigure}[b]{0.5\linewidth}
    \centering
    \includegraphics[scale =1]{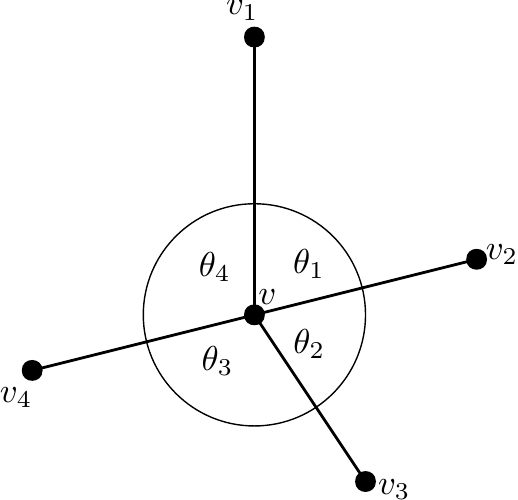} 
    \caption{An example of a star $S_4$ in the plane.}
    \label{fig:star4fig}  
  \end{subfigure}
  \hspace{4ex}
  \begin{subfigure}[b]{0.5\linewidth}
    \centering
    \includegraphics[trim = 0 -0 0 0,clip,scale =1]{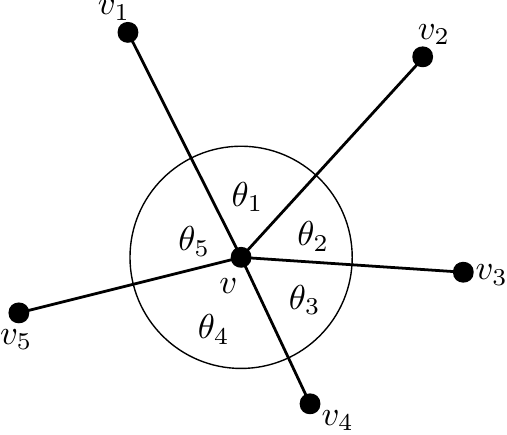} 
    \caption{An example of a star $S_5$ in the plane.} 
    \label{fig:star5fig}
  \end{subfigure}
    \label{fig:starsfig} 
    \caption{}
\end{figure}

Let $d(x,y)$ denote the Euclidean distance between the point $x$ and the point $y$. For $G$ to be an MST, we must satisfy the following necessary and sufficient conditions. \[\max\{d(v,v_i),d(v,v_j)\} \leq d(v_i,v_j) \quad \forall i,j \in \{1,2, \dots, D\},\] where $i \neq j$ and $D = 4$ or $5$. This is due to to the fact that the addition of any edge to the star creates a 3-cycle, and so we require that any new edge (which can only be between two radial points) be no shorter than both the other edges in the cycle.
Another way to say this is that for the angle $\theta_i$, the side opposite $\theta_i$ must be at least as long as both sides adjacent to $\theta_i$. For these conditions to hold, we require that each of the angles $\theta_i$ is at least $60^\circ$. If any $\theta_i$ is at least $90^\circ$, then the side opposite $\theta_i$ will be strictly longer than both of the sides adjacent to $\theta_i$. In order to allow angles strictly less than $90^\circ$ in our stars whilst still satisfying the conditions, we will make use of the following lemma.

\begin{lemma} \label{lem:starlem}
Let $x,y,z$ be side lengths of a triangle with $\theta$ as the angle between the sides of length $x$ and $y$, where $60^\circ \leq \theta < 90^\circ$. If $y = kx$ for some $k \geq 1$, then $z \geq y$ if and only if $k \leq  \frac{1}{2 \cos( \theta)}$.
\end{lemma}
\begin{proof}
Using the cosine rule, $z^2 = x^2 +k^2x^2$ - $2kx^2 \cos(\theta)$.
Hence
\begin{align}
z \geq y &\Leftrightarrow k^2x^2 \leq x^2 +k^2  x^2 - 2kx^2  \cos ( \theta )\\
&\Leftrightarrow 2k x^2 \cos( \theta) \leq x^2\\
&\Leftrightarrow k \leq  \frac{1}{2 \cos( \theta)}
\end{align}
\end{proof}

Lemma~\ref{lem:starlem} will be useful in the following section as it will allow us to choose appropriate side lengths for the star, assuming the angles are fixed. 
\subsection{Algorithm to Generate Star Points}
To make our test cases as general as possible, we wish to have an algorithm that can generate any possible set of star points whose MST is $S_4$ or $S_5$. In this section we describe such an algorithm, where the inputs to our algorithm are a length $L$ and an integer $D = 4$ or $5$, and the output is a random set of $D + 1$ points such that the MST of the points is a star $S_D$ whose longest edge is of length $L$.\\
The algorithm consists of four stages.
\begin{enumerate}
\item Angle stage.
\item Augmentation stage.
\item Scaling stage.
\item Rotation stage.
\end{enumerate}

In the angle stage, we randomly allocate the angles $\theta_1,\dots ,\theta_D$ such that $\sum_{i = 1}^D \theta_i = 360^{\circ}$, and each $\theta_i \geq 60^{\circ}$. After this stage, we have a set of $D$ radial points distributed around a centre point according to the allocated angles, where
each radial point is of unit distance from the centre point. These points satisfy
the conditions of being star points and hence the MST of these points is the star $S_D$. The augmentation stage involves increasing or decreasing the distance of each radial point from the centre point one at a time such that after each change of distance, or augmentation, the resultant points are still star points. The scaling stage involves scaling the distances of the radial points from the centre point by a constant factor (this will ensure that the longest edge in the star is of length $L$) and the rotation stage rotates all radial points around the centre point by a random angle.\\
To ensure that after each augmentation we are still left with star points, we use Lemma~\ref{lem:starlem} to give an allowable range of distances for each radial point from the centre point. For $i =1, \dots, D$, the \textit{left range} of the radial point $v_i$ is the interval 
\[ [2\cos (\theta_{i-1})\cdot d(v,v_i), \frac{d(v,v_i)}{2\cos (\theta_{i-1})}]\] if $\theta_{i-1} < 90^\circ$ and the set $\mathbb{R}_+$ otherwise, where we let $v_0 = v_D$ and $\theta_0 = \theta_D$. Similarly, the \textit{right range} of $v_i$ is the interval
\[ [2\cos (\theta_{i})\cdot d(v,v_{i+1}), \frac{d(v,v_{i+1})}{2\cos (\theta_{i})}]\] if $\theta_{i} < 90^\circ$ and the set $\mathbb{R}_+$ otherwise, where we let $v_{D+1} = v_1$. The \textit{allowable range} of $v_i$ is the intersection of its left range and right range. By applying Lemma~\ref{lem:starlem}, it can be seen that the allowable range of $v_i$ is the largest possible set of distances that $v_i$ can be from the centre point whilst ensuring the conditions of being star points are satisfied, maintaining all angles, and keeping the distances of its neighbouring radial points from the centre point fixed. In a single augmentation, a radial point is moved within its allowable range of distances from the centre point whilst keeping all angles fixed. A \textit{round} of augmentation
is complete after each radial point in clockwise order has been augmented. Our algorithm will repeat this process so that it will use $D - 2$ rounds of augmentation.\\
It can be seen that no stage of the algorithm will have the points violating the star point conditions, hence the output of the algorithm will be a valid set of star points. What remains to be shown is that the algorithm can generate every possible set of star points such that the longest edge in the MST of the points is of length $L$.\\ Suppose we had an arbitrary star $G = S_D$ embedded in the plane, where $D$ is either $4$ or $5$. We show that $G$ can be obtained as an output of our algorithm. Let $\theta_1, \dots , \theta_d$ be the angles of the star $G$. Clearly, star points with these angles can be obtained via the angle stage and the orientation of the angles can be set to match those of $G$ via the rotation stage. Hence, if we let $v_1, \dots , v_D$ be the radial points of $G$, then we can obtain via the algorithm a set of star points $S$, such that the radial points of $S$ correspond to the radial points of $G$ and corresponding angles are equal. To show that the distances of the radial points of $G$ from the centre point can be matched by applying the augmentation and scaling stages to $S$, we will show the reverse; that by applying the augmentation and scaling stages to the points of $G$, we can obtain a set of
star points whose radial distances are all unitary.
\begin{lemma} \label{lem:starlem2}
Given an arbitrary set of $D+1$ star points $S^*$, where $D$ is either $4$ or $5$, a set of corresponding star points $S$ can be obtained through $D-2$ rounds of augmentation and an application of the scaling stage such that each radial point of $S$ is of unit distance from the centre point.
\end{lemma}
\begin{proof}
Clearly, if we can apply the augmentation stage to $S^*$ to obtain a set of star points $S'$ such that the radial distances of $S'$ are uniform, then we can use the scaling stage to obtain $S$ from $S'$. Hence we will show that we can obtain $S'$ from $S^*$ via augmentations. Also note that not changing the distance of a radial point is considered to be a valid augmentation.\\
Given a set of star points $\hat{S}$, let $L_{\max}^{\hat{S}}$ denote the length of the longest distance between a radial point and the centre point in $\hat{S}$ and let $L_{\min}^{\hat{S}}$ denote the length of the shortest distance between a radial point and the centre point in $\hat{S}$. Assume that $v_0 = v_D$ and $v_{D+1} = v_1$. For a radial point $v_i$, its clockwise neighbour $v_{i+1}$, and its anti-clockwise neighbour $v_{i - 1}$, we describe three different states in which we will perform an augmentation of $v_i$. See Figure~\ref{fig:starsits} for an illustration of these states.

\begin{figure}[ht!]
\centering  
  \begin{subfigure}[b]{0.3\linewidth}
    \centering
    \includegraphics[scale =1]{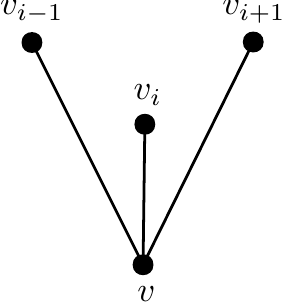} 
    \caption{The first state for augmentation.} 
  \end{subfigure}
  \hspace{1ex}
  \begin{subfigure}[b]{0.3\linewidth}
    \centering
    \includegraphics[trim = 0 -0 0 0,clip,scale =1]{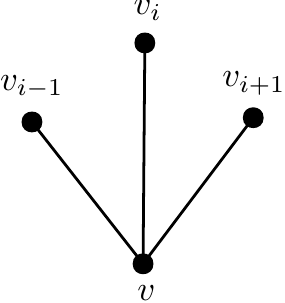} 
    \caption{The second state for augmentation.} 
  \end{subfigure}
  \hspace{1ex}
  \begin{subfigure}[b]{0.3\linewidth}
    \centering
    \includegraphics[trim = 0 -0 0 0,clip,scale =1]{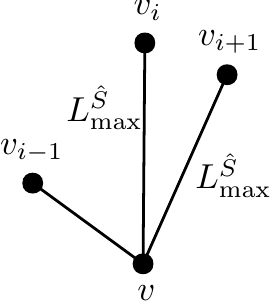} 
    \caption{The third state for augmentation.} 
  \end{subfigure} 
    \caption{The three states in which augmentation is performed in Lemma~\ref{lem:starlem2}.}
    \label{fig:starsits}
\end{figure}

\begin{enumerate}
\item If $d(v, v_i) < \min\{ d(v, v_{i - 1}), d(v, v_{i+1})\} $, then we augment $v_i$ so that $d(v, v_i) = \min\{ d(v, v_{i - 1}), d(v, v_{i+1})\} $.
\item If $d(v, v_i) > \max\{ d(v, v_{i-1}), d(v, v_{i+1})\} $, then we augment $v_i$ so that  $d(v, v_i) = \max\{ d(v, v_{i -1}), d(v, v_{i+1})\}$.
\item If  $d(v, v_i) > d(v,v_{i-1})$ and $d(v, v_i) = d(v, v_{i+1}) = L_{\max}^{\hat{S}}$, where $\hat{S}$ is the current set of star points, then we augment $v_i$ so that $d(v, v_i) = d(v,v_{i-1})$.
\end{enumerate}

Note that these augmentations are valid as $v_i$ remains within its allowable range after being augmented. Our procedure for transforming $S^*$ into $S'$ is as follows. First, augment one at a time each $v_i \in \{v_1,\dots , v_D\} $ that is in the first or second state, according to the augmentation rules described above, and let $S_1$ be the resultant set of star points after a single round of augmentation. Note that after this round of augmentation, no $v_i$ should be in the first or second state as augmenting some $v_i$ according to the augmentation rules should not not cause some $v_j$ to be in the first or second state if it was not in that state previously. Also note that after this round of augmentation, there will be at least two consecutive radial points whose distances from the centre point are both $L_{\max}^{S^1}$, as well as at least two consecutive radial points whose distances from the centre point are both $L_{\min}^{S^1}$, since otherwise there would be some $v_i$ in the first or second state. For the remaining augmentations, we consider the cases for $D = 4$ and $D = 5$ separately.\\
Suppose $D = 4$. Clearly if $L_{\max}^{S^1} = L_{\min}^{S^1}$, then the distances of all radial points is uniform and we can finish with $S' = S^1$. Otherwise, we have a consecutive pair of radial points with distances equal to $L_{\min}^{S^1}$ followed by a pair of radial points with distances equal to $L_{\max}^{S^1}$ (see Figure~\ref{fig:star4ex}). Thus we will have a radial point $v_i$ in the third state and we can perform the corresponding augmentation, after which $v_{i+1}$ will be in the second state and so we can perform another augmentation. After both these augmentations, the distances of all radial points from the centre point will be $L_{\min}^{S^1}$ and hence we have arrived at $S'$ after two rounds of augmentation.

\begin{figure}[ht!]
\centering
 \includegraphics[scale =0.8]{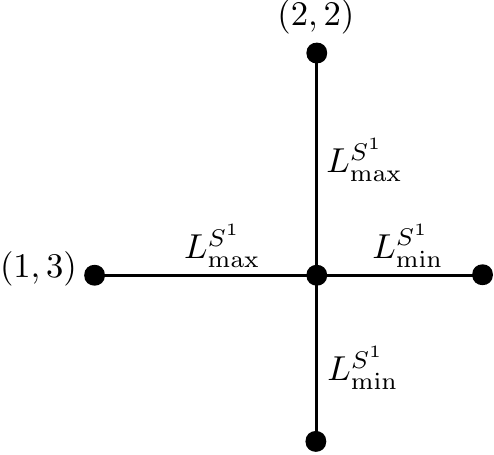}
 \caption{A possible set of star points $S^1$, where $|S^1| = 5$, after a single round of augmentation. An ordered pair $(i,j)$ labelling a radial point $v$ indicates that the augmentation of state $j$ will be used on $v$ at iteration $i$.}
 \label{fig:star4ex} 
\end{figure}

Now consider the case in which $D = 5$. Again, if $L_{\max}^{S^1} = L_{\min}^{S^1}$, then we are done, so we assume that $L_{\max}^{S^1} > L_{\min}^{S^1}$. If there are no distances of radial points from the centre other than $L_{\max}^{S^1}$ and $L_{\min}^{S^1}$, then we either have two or three consecutive radial points whose distances are $L_{\max}^{S^1}$. If there are two of these points, then similar to the case described above for $D = 4$, we perform the augmentation of the third state on the first of these points and the augmentation of the second state on the second point, resulting $S'$ (see Figure~\ref{fig:star5exa}). If there are three points whose distance is $L_{\max}^{S^1}$, then we perform the augmentation of the third state on the first two points in order, and the augmentation of the second state on the last point, to obtain $S'$ (see Figure~\ref{fig:star5exb}). Finally, suppose there is a radial distance in $S^1$ other than $L_{\min}^{S^1}$ or $L_{\min}^{S^1}$, i.e., there is a point $v_i \in S^1$ such that $L_{\min}^{S^1} < L' := d(v, v_i) < L_{\max}^{S^1}$. Hence we have two radial points of distance $L_{\min}^{S^1}$ from the centre point, two of distance $L_{\min}^{S^1}$, and one of distance $L'$ (see Figure~\ref{fig:star5exc}). Without loss of generality, we will assume that $v_{i+1}$ and $v_{i+2}$ are of distance $ L_{\max}^{S^1}$. The radial point $v_{i+1}$ is in the third state so we augment it accordingly, allowing us to augment $v_{i+2}$ according the rules of the second state. Let $S^2$ be the set of star points obtained performing these two augmentations. Hence we will have that $L_{\min}^{S^2} = L_{\min}^{S^1}$ and $L_{\max}^{S^2} = L'$, where the distances of $v_i$, $v_{i+1}$ and $v_{i+2}$ from the centre point are all $L'$. From this we can perform the augmentation of the third state on $v_i$ and $v_{i+1}$ in order, which leaves $v_{i+2}$ in the second state. Thus after augmenting $v_{i+2}$, all radial distances will be equal to $L_{\min}^{S^1}$ and we have arrived at $S'$ after three rounds of augmentation.
\end{proof}

\begin{figure}[ht!]
\centering  
  \begin{subfigure}[b]{0.3\linewidth}
    \centering
    \includegraphics[trim = -30 -0 0 0,scale =0.7]{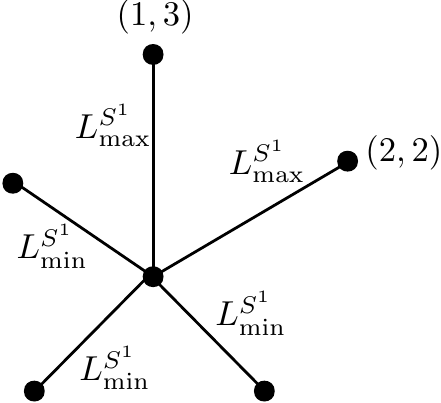}
     
    \caption{}
    \label{fig:star5exa} 
  \end{subfigure}
  \hspace{2ex}
  \begin{subfigure}[b]{0.3\linewidth}
    \centering
    \includegraphics[trim = 0 0 0 0,scale =0.7]{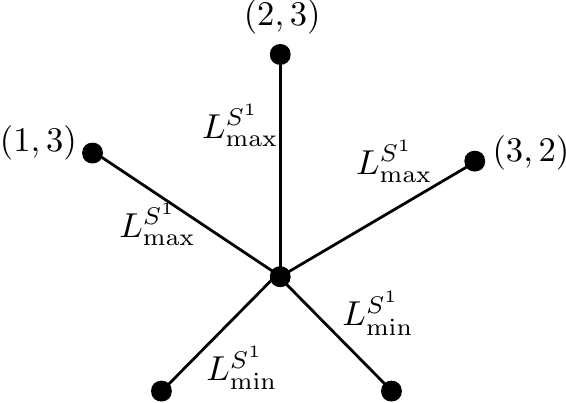} 
    \caption{}
    \label{fig:star5exb} 
  \end{subfigure}
  \hspace{2ex}
  \begin{subfigure}[b]{0.3\linewidth}
    \centering
    \includegraphics[trim = -7 -0 0 0,scale =0.7]{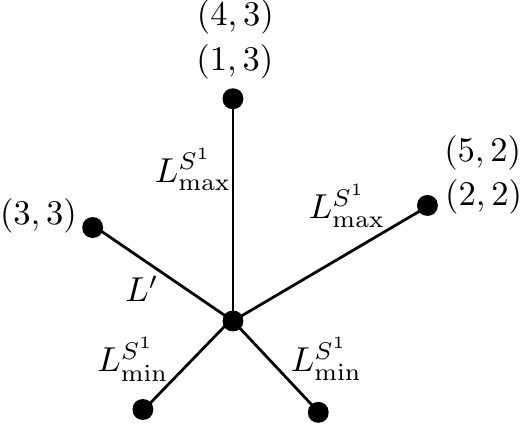} 
    \caption{}
    \label{fig:star5exc} 
  \end{subfigure} 
    \caption{Possible sets of star points $S^1$, where $|S^1| = 6$, after a single round of augmentation. An ordered pair $(i,j)$ labelling a radial point $v$ indicates that the augmentation of state $j$ will be used on $v$ at iteration $i$.}
    \label{fig:star5ex}
\end{figure}

Our test instances, which we refer to as \textit{special instances}, were created by generating star points such that approximately $10\%$ of the vertices of an instance were centre points of of a generated $S_4$, and approximately $5\%$ were centre points of a generated $S_5$. These proportions were chosen so that there would be enough points of high degree to make the instances interesting, whilst also most likely maintaining the ranking order of the most frequently occurring degrees, i.e., we would expect that the number of points of degree 4 be less than the number of points of degree 1,2, and 3 respectively, and we would expect the number of points of degree 5 be less than the number of points of degree 1,2,3, and 4 respectively. Our approach to placing the star points was to randomly construct an empty square subgrid within a 10000x10000 grid and then generate a star such that the length of the longest edge in the star was strictly less than half the length of the subgrid. The star points are then placed within the subgrid such that the centre point lies at the centre of the subgrid, and then the subgrid is ``blocked out" so that no more points can be positioned within it. This ensures that the star points will induce the appropriate star in the MST. The next star is placed in a square subgrid in the remaining space and this process continues until all star points have been generated and placed. Finally, the remaining points are placed uniformly randomly in the remaining space.\\
Although the special instances are somewhat contrived, they do ensure that the MSTs of the point sets have a relatively significant number of points of high degree and so can be used to test algorithms for $\delta = 3$ and $4$. 

\section{Local Edge Swap Approximation Algorithms}
A \textit{local edge swap} algorithm for the $\delta$-MST problem is any edge swapping algorithm in which the swaps are performed in a way such that the edge being swapped in will always share a common vertex with the edge being swapped out. In this section we will outline the local edge swap algorithms for the $\delta$-MST problem given by  Khuller, Raghavachari and Young \cite{khuller1996low} and Chan \cite{chan2003euclidean}. Unlike most of the other algorithms that we implemented, these algorithms are approximation algorithms with known upper bounds for their performances. Another advantageous feature of these algorithms are their running times, which are $O(n\log n)$ in the worst case.

\subsection{Khuller, Raghavachari and Young 1.5-factor 3-MST Algorithm}
Of the three local edge swap algorithms we explored, the algorithm of Khuller, Raghavachari and Young \cite{khuller1996low} is easiest to describe. The algorithm, which we refer to as the KRY algorithm, is able to produce a 3-MST for any point set in a metric space such that the total weight of the edges in the output tree is no worse than 1.5 times the total weight of an MST for the point set. It works by starting with a rooted MST $T$ for the input point set in the plane, to which it recursively applies local edge swaps in the following manner. If $v$ is the current root of $T$ with children $v_1,v_2, \dots, v_k$, then the edges $(v,v_2), \dots, (v,v_k)$ are replaced by a path through the vertices $v_1,v_2, \dots, v_k$. It then recursively applies the algorithm to each of the subtrees rooted at $v_1,v_2, \dots, v_k$ in turn, which we denote by  $T_{v_1},T_{v_2}, \dots, T_{v_k}$ respectively. See Figure~\ref{fig:ChanPaperKhullerFig} for an illustration. Let $w(T)$ denote the total weight of the tree $T$, and let $w(u,v)$ denote the weight of the edge $(u,v)$. We describe the KRY algorithm as Algorithm~\ref{alg:khuller}.

\begin{algorithm}[htb]
\caption{: KRY} \label{alg:khuller}
\begin{tabbing}
  ....\=....\=....\=....\=................... \kill \\ [-2ex]
\textbf{Input:} A rooted tree $T$ over a point set $P$ with root $v$ and a partially built\\ solution $T^*$.\\
\\
\textbf{if} $v$ has at least one child\\
\> Let the children of $v$ be $v_1, \dots, v_k$, where $k$ is the number of children of $v$,\\
\> such that $\sum_{i = 1}^{k-1} w(v_i, v_{i+1})$ is minimum if $k>1$.\\
\> Add the edge $(v,v_1)$ to $T^*$\\
\> Perform Algorithm~\ref{alg:khuller} with $T:=T_{v_1}$ as input.\\
\> \textbf{if} $k \geq 2$\\
\>\> \textbf{for} $i=1,\dots, k-1$\\
\>\>\> Add the edge $(v_i,v_{i+1})$ to $T^*$\\
\>\>\> Perform Algorithm~\ref{alg:khuller} with $T:=T_{v_{i+1}}$ as input.
\end{tabbing}
\end{algorithm}

\begin{figure} 
        \centering
        \includegraphics[scale=0.7]{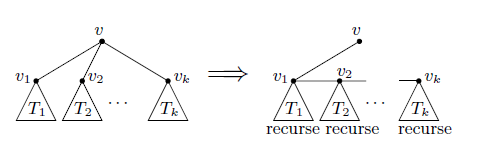}
        \caption{The illustration of the Khuller, Raghavachari and Young algorithm \cite{chan2003euclidean}.}
\label{fig:ChanPaperKhullerFig}        
\end{figure}

Although Algorithm~\ref{alg:khuller} uses a permutation of the children of $v$ that minimises the total weight of the path through the children, we also consider a version of the algorithm that minimises the bottleneck length of the path instead. We denote this bottleneck version of the KRY algorithm as the KRY-B algorithm. For both algorithms, the length of the longest edge in the outputted tree will be no worse than twice that of the MST \cite{andersen2016minimum}.

\subsection{Chan's 1.1381-factor 4-MST Algorithm} \label{sec:Chan4}
Another recursive local edge swap approximation algorithm is Chan's $4$-MST algorithm which produces a $4$-MST for a point set in the Euclidean plane such that the total weight of the outputted tree is no worse than 1.1381 times the total weight of the MST \cite{jothi2009degree}, and the longest edge in the outputted tree is no worse than 1.7321 times that of the MST \cite{andersen2016minimum}. Chan's algorithm works in a similar fashion to the KRY in that it performs local edge swaps recursively over rooted subtrees, however the edge swaps do not create a path over the child vertices of the current root as they did with KRY. The recursive steps of the algorithm are illustrated in Figure~\ref{fig:ChanPaperFig}. Choosing a good permutation of the children is necessary to achieve the 1.1381 bound and the details of the algorithm can be found in \cite{chan2003euclidean}. We refer to this algorithm as the Chan4 algorithm.

\begin{figure}[h] 
        \centering
        \includegraphics[scale=0.4]{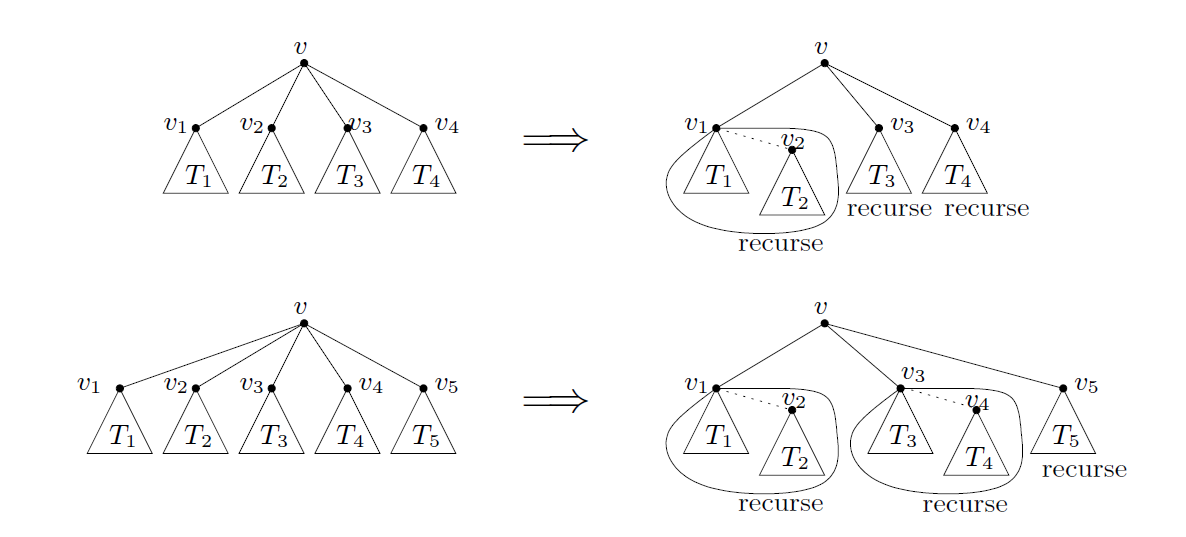}
        \caption{The illustration of Chan's 4-MST algorithm \cite{chan2003euclidean}.}
\label{fig:ChanPaperFig}        
\end{figure}

\subsection{Chan's 1.402-factor 3-MST Algorithm}
Chan also gives a 1.402-factor approximation algorithm for 3-MST problem in the Euclidean plane which uses similar local edge swap ideas as the previous algorithm for the 4-MST, although the algorithm itself is considerably more complicated. As such, we will omit any description of the algorithm. We refer to this algorithm as the Chan3 algorithm.

\section{Generalised MST Edge Swap Algorithms}
In this section we wish to generalise the $3$-MST algorithm of Khuller et al \cite{khuller1996low} and the $3$-MST and $4$-MST algorithms of Chan \cite{chan2003euclidean} to a general edge swap algorithm framework. Before we describe this approach, we first define some concepts.\\
Given a graph $G = (V,E)$ and an edge $e$ whose endpoints are in $V$, we let $G + e$ denote the graph $(V,E \cup \{e\})$, i.e, $G+e$ is the graph obtained by adding the edge $e$ to $G$. Similarly, we let $G - e$ denote the graph $(V,E \setminus \{e\})$, i.e, $G-e$ is the graph obtained by removing the edge $e$ from $G$. Let $T = (V,E')$ be a spanning tree for a graph $G = (V,E)$, where $E' \subset E$. Let $u$ be an edge in $E$ that is not in $T$. Hence the graph $T +u$ will contain a unique cycle $C$. Let $v$ be an edge of $C$ such that $v \neq u$. Hence if we remove $v$ from $T$ to obtain the graph $T' := (T+u) -v$, then $T'$ will be a spanning tree distinct from $T$. We say that any $T'$ obtained this way from $T$ is a \textit{neighbour} of $T$, and we let $N(T)$ denote the set of all neighbours of $T$, which we refer to as the \textit{neighbourhood} of $T$. An \textit{edge swap} is the operation that transforms any spanning tree $T$ into one of its neighbours $(T+u) -v$, and we say that $v$ is the edge that is \textit{swapped out} from $T$ and that $u$ is the edge that is \textit{swapped in}. If $u$ and $v$ are both incident to a common vertex, then we say that the edge swap is \textit{local}. The following lemma gives us a bound for the size of a neighbourhood of a spanning tree.
\begin{lemma}\label{lem:|N(T)|}
Given a spanning tree $T = (V,E)$, $|N(T)| = O(|V|^3)$.
\end{lemma}
\begin{proof}
The total number of edges that can be swapped into $T$ is $|V|(|V|-1)/2 -(|V| - 1) = O(|V|^2)$. After an edge $e$ is swapped into $T$, the number of edges in the unique cycle of $T +e$ is $O(|V|)$. Hence the total number of neighbours of $T$ is $O(|V|^3)$.
\end{proof}
The aforementioned algorithms of Khuller et al and Chan both start with an MST and use a sequence local edge swaps to eventually obtain a feasible $\delta$-MST. Furthermore, these algorithms have additional rules for choosing swaps so that it can be guaranteed that the algorithm obtains a feasible $\delta$-MST whose total weight is within some constant factor of that of the initial MST. For our approach we will incorporate the core idea of these algorithms, however we will not restrict ourselves to local edge swaps and instead shall consider arbitrary edge swaps. Given a complete graph $G$ for a given point set in the plane whose cardinality is at least three, we can outline our procedure as follows.

\begin{enumerate}
\item Find an MST $T$ for $G$.
\item Let $T$:= a neighbour of $T$ ($T$ will have a non-empty neighbourhood since $T \neq G$).
\item Repeat the previous step until a desired feasible solution is obtained.
\end{enumerate}

This outline omits much of the detail required for a functioning algorithm as we have not given specific termination criteria, nor a way of ensuring that the method will not infinitely cycle through a sequence of neighbours. There are multiple ways of addressing these details and we will describe some approaches in the subsequent sections.

\subsection{Feasibility Local Search}
The idea of the following algorithm is to have an algorithm that prioritises feasibility above all other criteria when searching through neighbours. As such, when it comes to choosing the next neighbour of the current spanning tree, this algorithm will choose a neighbour that is the closest to being feasible of all trees in the neighbourhood. We will refer to this algorithm as the \textit{feasibility local search} algorithm due to its connection to the local search optimisation algorithm.\\
In order to describe this algorithm, we must have some measure of feasibility. Let $d_G(v)$ denote the degree of the vertex $v$ in the graph $G$. Given a spanning tree $T = (V,E)$ and degree bound $\delta$, we let $f(T)$ be the \textit{feasibility error} of $T$, where \[f(T):= \sum_{v \in V} \max\{d_T(v) - \delta, 0\}.\] Hence the spanning tree $T$ is feasible if and only if it has a feasibility error of 0, and having such an error will be the stopping criteria of our algorithm. To ensure that the algorithm does not cycle through a sequence of neighbours, the algorithm will only choose a neighbour which has a feasibility error that is strictly less than the current spanning tree. We prove that such a neighbour always exists for every iteration in the following lemma.
\begin{lemma}
Let $T$ be spanning tree with $f(T) > 0$. There exits a spanning tree $T' \in N(T)$ such that $f(T') < f(T)$. 
\end{lemma}
\begin{proof}
Since $f(T) > 0$, there exists a vertex $v$ in $T$ such that $d_T(v) > \delta$. If we remove an edge $e$ incident to $v$, then the resultant graph $T-e$ has two connected components, denoted $T_1$ and $T_2$. Both of these connected components are trees, and hence have leaves (vertices of degree one). Let $v_1$ be a leaf of $T_1$, let $v_2$ be a leaf of $T_2$ and let $e' = (v_1,v_2)$. Note that $v \neq v_1$ and $v \neq v_2$ since $\delta \geq 2$. If we add the edge $e'$ to the current graph, then the resultant graph $T' = T -e +e'$ is a spanning tree. Moreover $f(T') < f(T)$ since we removed an edge incident to $v$ to decrease its degree, and the degrees of $v_1$ and $v_2$ in $T'$ are both 2 which is no more than $\delta$.   
\end{proof}

We describe the feasibility local search algorithm (FLS) as Algorithm~\ref{alg:fls}. Let $G$ be the input graph with $n$ vertices. The maximum number of times the while loop will be executed is $f(T)$, where $T$ is the initial MST. Since $T$ is in the Euclidean plane, we know that the degree of every vertex of $T$ is at most 6, hence $f(T) = O(n)$. Note that to obtain a neighbour of $T$ with less feasibility error than $T$, we must remove an edge from a vertex in $T$ with degree stricty greater than $\delta$. Hence there are $O(n)$ possible edges to delete since the degree of each vertex is constant. Removing an edge separates the graph into two subtrees $T_1$ and $T_2$. To minimise the feasibility error of the new tree, after removing an edge we should then add an edge between $T_1$ and $T_2$ such that the endpoints of this edge have degree strictly less than $\delta$ in their respective subtrees. We can search the subtrees using a depth-first-search to find such a pair of vertices in $O(n)$ time. Using the fact that an MST can be found in under cubic time \cite{pettie2002optimal}, we can obtain the complexity of the algorithm.

\begin{lemma}\label{lem:complxFLS}
For an input graph with $n$ vertices, the complexity of Algorithm~\ref{alg:fls} is $O(n^3)$.
\end{lemma}

\begin{algorithm}[tb]
\caption{: FLS} \label{alg:fls}
\begin{tabbing}
  ....\=....\=....\=....\=................... \kill \\ [-2ex]
\textbf{Input:} A complete graph $G$ and degree bound $\delta$.\\
\\
Let $T$ be an MST for $G$.\\
\textbf{while} $f(T) > 0$,\\
\> Let $T' \in N(T)$ be the neighbour of $T$ such that $f(T')$ is minimum.\\
\> $T: = T'$.\\
\textbf{return} $T$.
\end{tabbing}
\end{algorithm}

Algorithm~\ref{alg:fls} does not take into account the weight of the neighbours it searches through; it merely greedily chooses the neighbour that is the best in terms of feasibility. We can easily modify Algorithm~\ref{alg:fls} so that of the neighbours with smaller feasibility error, the one with the smallest weight is chosen. This modification is presented as Algorithm~\ref{alg:fls2}, which we refer to as the \textit{feasibility weight local search} algorithm (FWLS). The worst-case time complexity of the algorithm is given in the following lemma.

\begin{lemma} \label{lem:complxFWLS}
For an input graph with $n$ vertices, the complexity of Algorithm~\ref{alg:fls2} is $O(n^4)$.
\end{lemma}

\begin{proof}
As mentioned previously, $f(T) = O(n)$. Combining this with Lemma~\ref{lem:|N(T)|} and the fact that given the weight of $T$, we can determine the weight of a neighbour of $T$ in constant time, we obtain a worst-case complexity for Algorithm~\ref{alg:fls2} of $O(n^4)$.
\end{proof}

\begin{algorithm}[tb]
\caption{: FWLS} \label{alg:fls2}
\begin{tabbing}
  ....\=....\=....\=....\=................... \kill \\ [-2ex]
\textbf{Input:} A complete graph $G$ and degree bound $\delta$.\\
\\
Let $T$ be an MST for $G$.\\
\textbf{while} $f(T) > 0$,\\
\> Let $N:= \{T' \in N(T):f(T') < f(T)\}$\\
\> Let $T^*$ be the tree in $N$ such that $w(T^*)$ is minimum.\\
\> $T: = T^*$.\\
\textbf{return} $T$.
\end{tabbing}
\end{algorithm}

Note that we can alter Algorithm~\ref{alg:fls2} to focus on bottleneck length by replacing $w(T^*)$ with $b(T^*)$, where $b(T)$ denotes the length of a longest edge in the tree $T$. We refer to this modified version of FWLS as FWLS-B. Also note that for both of FLS and FWLS, we are requiring that the next chosen neighbour must have a feasibility error strictly smaller than that of the current spanning tree. Whilst this condition helps to ensure that the algorithm will terminate at a feasible solution, it is not a necessary condition in order to have a practical edge swap algorithm. In fact, neither the $3$-MST algorithm of Khuller et al., nor the $4$-MST algorithm of Chan have this condition. In the next two sections, we describe algorithms which also do not require this condition.



\subsection{Bi-Criteria Local Search}
The bi-criteria local search algorithm (BCLS) is similar to FWLS with one key exception; whilst FWLS searches for the smallest weight neighbour of the current tree $T$ from the neighbourhood $N:= \{T' \in N(T):f(T') < f(T)\}$, BCLS searches for the smallest neighbour from the neighbourhood $\hat{N}~:=~ \{T' \in N(T):f(T') \leq f(T)\}$. Clearly, a local search using such a neighbourhood may reach a point in which the current tree $T$ has non-zero feasibility error but the smallest weight neighbour of $T$ in $\hat{N}$ has the same weight and feasibility error as $T$, which may result in the algorithm being stuck in an infeasible local minimum. In order to mitigate this, whenever the smallest weight neighbour $T^*$ in $\hat{N}$ has weight equal to $w(T)$, with $f(T^*) = f(T)$, we replace $\hat{N}$ with $N$; the neighbourhood from FWLS with the strict inequality. Thus the feasibility error will be reduced in the next iteration, after which we can go back to using $\hat{N}$ again. The details of BCLS are given in Algorithm~\ref{alg:bcls}.

\begin{algorithm}[tb]
\caption{: BCLS} \label{alg:bcls}
\begin{tabbing}
  ....\=....\=....\=....\=................... \kill \\ [-2ex]
\textbf{Input:} A complete graph $G$ and degree bound $\delta$.\\
\\
Let $T$ be an MST for $G$.\\
\textbf{while} $f(T) > 0$,\\
\> Let $\hat{N}:= \{T' \in N(T):f(T') \leq f(T)\}$\\
\> Let $T^*$ be the tree in $\hat{N}$ such that $w(T^*)$ is minimum.\\
\> \textbf{if} $w(T^*) = w(T)$ and $f(T^*) = f(T)$,\\
\> \> Let $N:= \{T' \in N(T):f(T') < f(T)\}$\\
\> \> Let $T^*$ be the tree in $N$ such that $w(T^*)$ is minimum.\\
\> $T: = T^*$.\\
\textbf{return} $T$.
\end{tabbing}
\end{algorithm}

\subsection{Diminishing Neighbourhood Local Search} \label{sec:DNLS}
In this section, we describe an edge swap algorithm such that the sequence of neighbours it chooses may not be non-increasing in terms of feasibility error. As such, we will need additional rules to ensure that such an algorithm will eventually reach a feasible solution within a reasonable amount of time. The motivation of this approach is to have some way of generalising the edge swapping methods seen in the previously described constant factor approximation algorithms, namely KRY and Chan4, into a less restrictive local search algorithm.\\
One observation of the KRY and Chan4 algorithms is that once a node is``visited", i.e., becomes the root of a subtree, its feasibility error is never increased in any subsequent iterations, whereas nodes that are not yet visited may have their feasibility increased in a future iteration. With that observation in mind, we developed the following rule to be used by our new algorithm; once a node has had its feasibility error decrease as the result of an edge swap, its feasibility error may not increase in any future iterations. We will show that this rule is sufficient to create an edge swap algorithm that terminates at a tree of zero feasibility error using polyhedral description of such an algorithm.\\
First we define an integer programming formulation for the MST. For a set of vertices $S \subseteq V$, let $E(S)$ denote the set of edges in $E$ that have both endpoints in $S$, i.e., $E(S) = \{(e_1,e_2) \in E : e_1,e_2 \in S\}$. For a vertex $v \in V$, let $E(v)$ denote the set of edges in $E$ that have $v$ as an endpoint, i.e., $E(v) = \{(v,e_1) \in E : e_1 \in V \}$. For each $e \in E$, let $x_e$ be a binary decision variable that is unit valued when the edge $e$ is in the spanning tree, and zero otherwise. The integer program for the MST problem is\\

\begin{alignat}{3}
\text{minimise }   \sum_{e \in E} \ \ &w(e)x_e \label{eqn:minsumobj1}  \\ 
\text{s.t. }  \sum_{e \in E} \ \ & x_e  =  |V| -1  \\
\sum_{e \in E(S)} &x_e  \leq |S| - 1 & \quad \forall S \subset V\\
& x_e \in \{0,1\} & \quad \forall  e \in E \label{eqn:minsumend1}
\end{alignat}

Note that the integrality conditions of the above formulation can be omitted since the linear programming relaxation will always yield an integer optimal solution \cite{edmonds1971matroids}. In fact, all basic feasible solutions are integer valued and hence represent spanning trees. Moreover, there is a bijection between the set of basic feasible solutions for the formulation and the set of spanning trees of the input set. By adding degree constraints to the integer program, the $\delta$-MST problem can be formulated as follows.

\begin{alignat}{3}
\text{minimise }   \sum_{e \in E} \ \ &w(e)x_e \label{eqn:minsumobj2}  \\ 
\text{s.t. }  \sum_{e \in E} \ \ & x_e  =  |V| -1  \\
\sum_{e \in E(S)} &x_e  \leq |S| - 1 & \quad \forall S \subset V\\
\sum_{e \in E(v)} &x_e  \leq \delta & \quad \forall  v \in V  \label{eqn:degreebounds}\\
& x_e \in \{0,1\} & \quad \forall  e \in E \label{eqn:minsumend2}
\end{alignat}

Unlike the MST formulation, the linear programming relaxation of the $\delta$-MST problem may yield fractional optimal solutions. Let $A$ be the polytope determined by the convex hull of feasible solutions to \ref{eqn:minsumobj1}~-~\ref{eqn:minsumend1}, and similarly, let $B$ be the polytope determined by the convex hull of feasible solutions to \ref{eqn:minsumobj2}~-~\ref{eqn:minsumend2}. Thus we have that $B \subseteq A$. The integer optimal points on the boundary of $B$ will be basic feasible solutions to \ref{eqn:minsumobj1}~-~\ref{eqn:minsumend1}, hence these boundary points of $B$ are also contained in the boundary of $A$.
\\~\\
Our goal to is to add successive cutting planes to $A$ to find an extreme point of $B$, where the cutting planes come from relaxations of the constraints \ref{eqn:degreebounds}. We outline the approach as follows.
\begin{enumerate}
\item Let $x$ be an integer optimal solution to the IP formulation $F$ := \ref{eqn:minsumobj1}~-~\ref{eqn:minsumend1}.\\
\item If $x$ is on the boundary of $B$ then finish, otherwise let $x'$ be a basic feasible solution of $F$ that is adjacent to $x$ on the boundary of $A$.\\
\item Add the constraint $\sum_{e \in E(v')} x_e  \leq \delta +c$ to $F$, for some $v' \in V$ and $c \in \mathbb{Z}_+$ such that $x'$ is still feasible in $F$, but now $x$ is infeasible in $F$.\\
\item Let $x := x'$ and go to Step 2.
\end{enumerate}
The final $x$ will be our desired solution to the $\delta$-MST formulation.\\
Extreme points of $A$ that are contained in $B$ will also be extreme points of $B$ since $B \subseteq A$. Hence, due to the stopping criteria, the above approach will eventually result in a basic feasible solution to the $\delta$-MST problem. The connection between this polyhedral algorithm and edge swapping algorithms can be seen by identifying the fact that if $T$ and $T'$ are neighbouring spanning trees, i.e., $T' \in N(T)$, then basic feasible solutions representing $T$ and $T'$ are adjacent in the polyhedron $A$.\\ 
For a solution $X$ in $A$, let $X(x_e)$ denote the value of $x_e$ in $X$. Let $X_1$ and $X_2$ be two neighbouring basic feasible solutions in $A$. Thus there exists $f,g \in E$ such that $X_1(x_g) = 1, X_2(x_g) = 0, X_1(x_f) = 0, X_2(x_f) = 1$, and $X_1(x_e) = X_2(x_e)$ for all $e \neq f,g$. Let $T_1$ and $T_2$ be the trees represented by $X_1$ and $X_2$ respectively. Then $T_2$ can be obtained from $T_1$ in a single edge swap by deleting the edge $g$ and adding the edge $f$, hence $T_2 \in N(T_1)$. Thus the polyhedral algorithm we described can be realised as an edge swap algorithm where we constrain the edge swaps we consider as the algorithm progresses.\\  
 
We describe one way of realising of this polyhedral approach as an edge swap algorithm. The chosen approach here is to restrict the set of vertices whose degrees can be increased by an edge swap. The set of such vertices decreases as the algorithm progresses until the algorithm arrives at a feasible solution. Hence we refer to our algorithm as the \textit{diminishing neighbourhood local search} (DNLS). Our algorithm works by partitioning the vertices of the input graph into \textit{locked}, \textit{semi-locked} and \textit{unlocked} vertices. Initially there is a single locked vertex, and all vertices are unlocked. A vertex only becomes locked if its current degree is strictly greater than $\delta+1$ and if its degree is then decreased after the algorithm performs an edge swap for the current iteration. Once a vertex is locked, its degree cannot be increased by any edge swap until its degree is no bigger than $\delta$, at which point it becomes semi-locked. A semi-locked vertex's degree can be increased again, but only while its degree is strictly less than $\delta$. A locked or semi-locked vertex can never become unlocked at a later stage of the algorithm. To avoid cycling, we require that every edge swap must reduce the degree of a locked vertex or create a new locked vertex from an unlocked one. We present this algorithm as Algorithm~\ref{alg:esdn}.

\begin{algorithm}[tb]
\caption{: DNLS} \label{alg:esdn}
\begin{tabbing}
  ....\=....\=....\=....\=................... \kill \\ [-2ex]
\textbf{Input:} A complete graph $G=(V,E)$, degree bound $\delta$.\\
Let $T := (V,E' \subseteq E)$ be an MST for $G$.\\
Let $L := \{\}$, $U := \{V\}, S := \{\}$ be the locked, unlocked, semi-locked vertices respectively.\\
\textbf{while} $T$ is not a feasible $\delta$-MST,\\
\> Let $N = \{T' \in N(T): T'\text{ is obtained through a single edge swap that decreases the degree}$\\
\> of a vertex with positive feasibility error but does not increase the degree of any locked \\
\>vertex or semi-locked vertex with degree $\delta$ $\}$.\\
\> Let $T^*$ be the tree in $N$ such that $w(T^*)$ is minimum.\\
\> \textbf{for each} vertex $v$ such that $d_T(v) > d_{T^*}(v)$,\\
\>\> \textbf{if} $v \in U$,\\
\>\>\>$U := U \setminus \{v\}$.\\
\>\>\> \textbf{if} $d_{T^*}(v) > \delta$,\\
\>\>\>\> $L := L \cup \{v\}$.\\
\>\>\> \textbf{else} $S:= S \cup \{v\}$.\\
\>\> \textbf{else if} $v \in L$ and $d_{T^*}(v) = \delta$,\\
\>\>\> $L := L \setminus \{v\}$.\\
\>\>\> $S := S \cup \{v\}$.\\
\> $T:= T^*$.\\
\textbf{return} $T$.
\end{tabbing}
\end{algorithm}

It can be seen from Algorithm~\ref{alg:esdn} that a locked vertex must have positive feasibility error and that the degree of locked vertices never increases as the result of an edge swap. Hence, in order to show that the algorithm terminates at a feasible $\delta$-MST solution, we prove the following.

\begin{lemma} \label{lem:locklem}
Each edge swap used by Algorithm~\ref{alg:esdn} either decreases the degree of a locked vertex, or decreases the number of unlocked vertices.
\end{lemma}
\begin{proof}
Suppose that we perform an edge swap that does not decrease the degree of a locked vertex. By the definition of $N$, the edge swap must decrease the degree of a vertex $v$ with positive feasibility error. Since semi-locked vertices have a feasibility error of 0, this implies that $v$ is an unlocked vertex. As a result of this swap, $v$ will become locked and since unlocked vertices are never created as a result of edge swaps, the total number of unlocked vertices will decrease.
\end{proof}

Since locked and unlocked vertices are the only possible vertices with positive feasibility error, and the maximum feasibility error of a vertex is bounded by a constant, Lemma~\ref{lem:locklem} implies that the total feasibility error of $T$ will eventually decrease to 0 in $O(n)$ iterations of the while loop. Since $N \subseteq N(T)$ we know from Lemma~\ref{lem:|N(T)|}, that $|N| = O(n^3)$. If a list of the degrees of each vertex is maintained, one can verify if an element of $N(T)$ is an element of $N$ in constant time. Similarly, the updating of the sets $U,L$ and $S$ can be performed in constant time. Thus the time complexity of Algorithm~\ref{alg:esdn} is also $O(n^4)$.\\
We now show that the DNLS algorithm is a realisation of the aforementioned polyhedral approach. The algorithm first starts with an MST $T$ for the point set and so we begin with an integer optimal solution as in Step 1. of the outline. Next, assuming $T$ is not already a $\delta$-MST, we peform an edge swap on $T$, which is equivalent to moving to a neighbouring basic feasible solution on the boundary of the polyhedron $A$, as in Step 2. Finally, note that each edge swap must reduce the degree of a vertex with positive feasibility error. Let $v$ be such a vertex for a given edge swap and let $d > \delta$ be its degree before the swap. Thus its degree becomes $d-1$ after the swap. Since $v$ had positive feasibility error before the swap, it must have been either a locked or unlocked vertex. Hence after the swap, $v$ will be a locked vertex. The rules of the algorithm will now prohibit the use of any edge swap that will increase the degree of $v$ above $d-1 \geq \delta$, which is equivalent to the using the constraint $\sum_{e \in E(\{v\})} x_e  \leq d-1$  as in Step 3. Thus the DNLS algorithm can be seen as a valid realisation of the polyhedral approach.

\section{$\delta$-Prim's and the Multistart Hillclimbing Procedure}
The algorithms in this section are based upon the approach of Knowles and Corne \cite{knowles2000new}. Starting with a modification of Prim's algorithm, known as \textit{$\delta$-Prim's}, the authors describe a technique which they call the \textit{randomised primal method} (RPM), which allows for the use of metaheursitcs to influence the edge choices of $\delta$-Prim's. In the following subsections, we describe $\delta$-Prim's algorithm as well as a use of RPM in conjunction with a multistart hillclimbing procedure as the metahueristic.
\subsection{$\delta$-Prim's Algorithm} \label{sec:dPrim}
This modification of Prim's algorithm was first introduced by Narula and Ho \cite{narula1980degree}. The conventional Prim's algorithm starts with a tree $T$ that contains a single vertex and during each iteration, the cheapest edge $e$ is chosen, where exactly one endpoint of $e$ is in $T$. The edge $e$ is added to $T$ and this process is repeated until $T$ is a spanning tree. $\delta$-Prim's algorithm operates in the same way but with the extra requirement that the endpoint of $e$ that is in $T$ must have degree strictly less than $\delta$. In this way we ensure that the final spanning tree satisfies the degree bound. We describe $\delta$-Prim's in Algorithm~\ref{alg:dprim}.

\begin{algorithm}[tbh]
\caption{: $\delta$-Prim's} \label{alg:dprim}
\begin{tabbing}
  ....\=....\=....\=....\=................... \kill \\ [-2ex]
\textbf{Input:} A graph $G=(V,E)$, degree bound $\delta$.\\
Let $V_T := \{v_0\}$ for some $v_0 \in V$.\\
Let $E_T := \{ \}$.\\
Let $T$ denote the graph $(V_T,E_T)$.\\
\textbf{while} $|E_T| < |V| - 1$,\\
\> Find the smallest weight edge $(u,v)$ such that $u \in V_T$, $v \notin V_T$, and $d_T(u) < \delta$.\\
\> Add $v$ to $V_T$.\\
\> Add $e$ to $E_T$.\\ 
\textbf{return} $T$.
\end{tabbing}
\end{algorithm}

\subsection{RPM and MHC}
The randomised primal method described by Knowles and Corne works in a similar way to $\delta$-Prim's, with the exception that the new edges added to the tree are not necessarily chosen greedily, but instead they are chosen according to a predefined table of values. Using the parlance of genetic algorithms, such a table is referred to as a \textit{tabular chromosome}. For a graph $G=(V,E)$ with vertices labelled $1, \dots, n$ and a degree bound $\delta$, a tabular chromosome for $G$ will have $n$ rows and $\delta$ columns. For each $i \in [1, n]$, $j \in [1,\delta]$, the entry in row $i$ and column $j$ is assigned an allele value $a(i,j) \in [1,n]$. The allele values indicate which edges the RPM version of $\delta$-Prim's considers in the following way: for each $i \in [1,n]$ let $L_i$ be the list of edges $(i,j) \in E - E_T, j \neq i$ sorted by weight. During each iteration, the algorithm only considers the edges $(i',j')$ to add to the current tree $T$, where $i' \in V_T$, $d_T(i') < \delta$ and $(i',j')$ is the $a(i',d_T(i'))$-th edge in $L_{i'}$. In the case where the $a(i',d_T(i'))$ is greater than $|L_{i'}|$, then the algorithm uses the last edge of $L_{i'}$ instead of the $a(i',d_T(i'))$-th edge. Hence the tabular chromosome provides a methodology for choosing edges during each iteration. We describe RPM in more detail as Algorithm~\ref{alg:rpm}.

\begin{algorithm}[tbh]
\caption{: RPM} \label{alg:rpm}
\begin{tabbing}
  ....\=....\=....\=....\=................... \kill \\ [-2ex]
\textbf{Input:} A labelled graph $G=(V,E)$, degree bound $\delta$, chromosome $a$.\\
Let $V_T := \{v_0\}$ for some $v_0 \in V$.\\
Let $E_T := \{ \}$.\\
Let $T$ denote the graph $(V_T,E_T)$.\\
\textbf{while} $|E_T| < |V| - 1$,\\
\> Let $E' = \{\}$\\
\> \textbf{For each} $i \in V_T$ with $d_T(i) < \delta$,\\
\>\> Let $L_i$ be a sorted list of edges in $\{(i,j):j \notin V_T \}$.\\
\>\> Add the $\min\{a(i,d_T(i)),|L_i|\}$-th edge of $L_i$ to $E'$.\\
\> Find the smallest weight edge $(u,v) \in E'$.\\
\> Add $v$ to $V_T$.\\
\> Add $e$ to $E_T$.\\ 
\textbf{return} $T$.
\end{tabbing}
\end{algorithm}

Note, that if we set all the allele values in the tabular chromosome to one, then Algorithm~\ref{alg:rpm} is equivalent to $\delta$-Prim's. By applying small local changes to a chromosome, a neighbourhood of chromosomes can be produced, where the value of each chromosome is given by the weight of the spanning tree output by the algorithm when given the chromosome as input. This allows for the application of metaheuristics in order to iteratively search chromosome neighbourhoods so that a low weight spanning tree can be found. For our analysis, we chose the multistart hillclimbing (MHC) algorithm described by Knowles and Corne as one such algorithm.\\
A traditional hillclimbing algorithm is a local search technique that starts with an initial solution and then chooses a member of the solution's neighbourhood and evaluates it. If the evaluated solution is better than the current solution, then it becomes the new current solution. This greedy process is repeated until some stopping criteria are met. The MHC algorithm is similar except that the search process can be occasionally restarted from a new solution when no improvements are made after a specified number of evaluations. The algorithm then continues in the same fashion until the stopping criteria are met, with the best solution found amongst all restarts of the algorithm being the final output.\\

Our specific approach to MHC is given by Algorithm~\ref{alg:mhc}, where RPM$(a,G,\delta)$ is the output to Algorithm~\ref{alg:rpm} with chromosome $a$, graph $G$ and degree bound $\delta$ as inputs. The initial chromosomes and neighbouring chromosomes are found using the exponential probability distribution described by \cite{knowles2000new} and we set $m$ and $r$ to be 5000 and 250 respectively. Note that our approach uses half of the number evaluations used by Knowles and Corne. This is because the repeated calls to the RPM algorithm make MHC computationally expensive for a large number of evaluations. We therefore chose a number of iterations that made the running time of MHC comparable to that of the other heuristics we examined.

\begin{algorithm}[tbh]
\caption{: MHC} \label{alg:mhc}
\begin{tabbing}
  ....\=....\=....\=....\=................... \kill \\ [-2ex]
\textbf{Input:} A labelled graph $G=(V,E)$, degree bound $\delta$, evaluation limit $m$, reset number $r$.\\
Let $a$ be initial tabular chromosome.\\
Let $T :=$ RPM$(a,G,\delta)$.\\
Let $T' := T.$\\
Let $m' := 1$.\\
Let $r': = 1$.\\
\textbf{while} $m' \leq m$,\\
\> Let $a' :=$ neighbouring chromosome of $a$.\\
\> Let $T'' := $ RPM$(a',G,\delta)$\\
\> \textbf{if} $w(T'') < w(T')$,\\
\>\> $T' := T''$.\\
\>\> $a := a'$\\
\>\> \textbf{if} $w(T) < w(T')$,\\
\>\>\> $T := T'$.\\
\> \textbf{else}\\
\>\> $r' := r' +1$.\\
\>\> \textbf{if} $r' > r$,\\
\>\>\> Reset $a$ to an initial tabular chromosome.\\
\>\>\>  $T' :=$ RPM$(a,G,\delta)$.\\
\> $m' := m' +1$.\\
\textbf{return} $T$.
\end{tabbing}
\end{algorithm}

\section{Approximation Algorithms using Eulerian and Hamiltonian Supergraphs of an MST} \label{sec:cube2}
In this section we describe three approximation algorithms for the $2$-MST/$2$-MBST that that are not based upon edge swaps. Instead, each of these algorithms starts with an MST of the point set, but adds edges to the tree to create either an Eulerian or Hamiltonian supergraph of the MST. The output of each of these algorithms is a Hamiltonian path of the point set that is obtained by traversing the supergraph.

\subsection{An Approximation Algorithm using Tree Edge Doubling}
A widely known 2-approximation algorithm for the metric TSP is the so-called \textit{double tree} algorithm \cite{rosenkrantz1977analysis}. Given a set of points $P$ in a metric space, the double tree algorithm works by duplicating every edge in an MST $T$ of $P$. This results in a spanning Eulerian multigraph $T_2$, since the degree of every vertex in $T_2$ is even. An Eulerian circuit $C$ can than be found in $T_2$, which can be reduced to a Hamiltonian cycle $C^*$ by \textit{short-cutting}. The short-cutting process naively removes any duplicates of vertices from $C$ to convert the Eulerian circuit into a Hamiltonian cycle $C^*$. We describe a modified version of the double tree algorithm for the 2-MST problem, which we refer to as the DT algorithm, as follows.
\begin{enumerate}
\item Find an MST for $P$, denoted $T$.
\item Duplicate every edge in $T$ to create a new graph $T_2$.
\item Find an Eulerian circuit $C$ in $T_2$.
\item Apply short-cutting to $C$ to create the Hamiltonian cycle $C^*$.
\item Delete the longest edge in $C^*$ to create the Hamiltonian path $P^*$. This is the desired approximation.
\end{enumerate}
By applying the triangle inequality, one can see that the total weight of $C^*$ is no greater than the total weight of $C$, the Eulerian circuit which contains every edge in $T_2$. Hence the total weight of $P^*$ is no more than twice the total weight of the starting MST, hence the DT algorithm is a 2-factor approximation algorithm for the 2-MST problem. Note that the short-cutting process may result in $C^*$ having a longer bottleneck edge than $C$, and so we do not have the same performance guarantee for the the 2-MBST problem.

\subsection{Christofides Algorithm}
Another approximation algorithm which uses an Eulerian supergraph of an MST is Christofides' $\frac{3}{2}$-approximation algorithm for the metric TSP \cite{christofides1976worst}. The algorithm works by starting with an MST $T$ of the point set $P$ as in the double tree algorithm. A supergraph $T_M$ of $T$ is created by finding a minimum weight perfect matching $M$ of the odd degree vertices of $T$ and adding the edges of $M$ to $T$. Note that the degree of every vertex in $T_M$ is even, so $T_M$ is Eulerian. The algorithm then proceeds in the same way as the tree double algorithm.\\
To convert the Hamiltonian cycle version of the algorithm into a Hamiltonian path version, Hoogeveen \cite{hoogeveen1991analysis} provides a modification of Christofides algorithm which maintains the $\frac{3}{2}$ performance guarantee. Let $V_\text{odd}$ be the set of odd vertices of the initial MST $T$. Let $u_1$ and $u_2$ be new dummy vertices, and for each vertex $v \in V_\text{odd}$, join $v$ to $u_1$ and join $v$ to $u_2$ by edges of zero length. Now instead of finding a minimum weight perfect matching for the vertices in $V_\text{odd}$, we instead find a minimum weight perfect matching $M$ of $V_\text{odd} \cup \{u_1,u_2\}$. We then add all edges of $M$ which do not have dummy vertices as endpoints to $T$ to obtain $T_M$. The graph $T_M$ will contain exactly two vertices, $v_1$ and $v_2$, of odd degree and will thus contain an Eulerian trail $P$ between $v_1$ and $v_2$. After applying short-cutting to $P$, we obtain the Hamiltonian path $P^*$ whose total weight is at most $\frac{3}{2}$ times the total weight of the optimal solution to the 2-MST problem. Note that once again we do not have the same performance guarantee for the 2-MBST due to the short-cutting process. We outline the algorithm as follows.
\begin{enumerate}
\item Find an MST for $P$, denoted $T$. Let $V_\text{odd}$ be the set of odd vertices of $T$.
\item Add dummy vertices $u_1$ and $u_2$ and zero length edges $(v,u_1),(v,u_2)$, where $v \in V_\text{odd}$, to $T$.
\item Find a minimum weight perfect matching $M$ of $V_\text{odd} \cup \{u_1,u_2\}$.
\item Add all edges of $M$ that do not have $u_1$ or $u_2$ as endpoints to $T$ to obtain $T_M$.
\item Find an Eulerian trail $P$ in $T_M$.
\item Apply short-cutting to $P$ to create the Hamiltonian path $P^*$. This is the desired approximation.
\end{enumerate}
 
\subsection{An Approximation Algorithm using Graph Cube}    
Finally, we describe an approximation algorithm for the $2$-MST problem that does not use short-cutting on an Eulerian circuit/trail in order to obtain a Hamiltonian path. Developed by the authors in \cite{andersen2016minimum}, this algorithm is based upon the idea that one can find a Hamiltonian path in the cube of a connected graph in polynomial time \cite{karaganis1968cube}. The cube of a graph $G$, denoted $G^3$, is the supergraph of $G$ such that the edge $(u,v)$ is in $G^3$ if and only if there is a path between $u$ and $v$ in $G$ with three or fewer edges. Starting with a point set $P$, the algorithm can be described as follows.
\begin{enumerate}
\item Find an MST for $P$, denoted $T$.
\item Compute $T^3$.
\item Find a Hamiltonian path in $T^3$. This is the desired approximation.
\end{enumerate} 
Assuming that $P$ lies within a metric space, one can apply the triangle inequality to show that the both the total weight and bottleneck length of the approximation is no worse than three times those of the MST for $P$. We refer to this algorithm as Cube2.
\section{Results}
We present the results of our computational experiments by graphing the average performances of the algorithms in terms of total weight and bottleneck lengths. For clarity, the graphs have been separated into the best and worst performers given by the results for the case when $n = 100$. The full tables of results can be found in the appendices, including the average running times of the algorithms in seconds. Note that we treat the running times of the algorithms as secondary considerations and so we do not claim that the implementations of the algorithms we tested are optimal with respect to time efficiency. With the exception of the uniformly random instances for $\delta = 3$, the results are given as averages over 30 test instances for each value of $n \in \{10,20, \dots, 100 \}$ for uniformly random instances and $n \in \{11,20, \dots, 100 \}$ for our special instances. Our tests were performed on a Dell PowerEdge R820 with four Xeon E5-4650 2.7GHz CPU's and 256GB of RAM.\\

\subsection{Degree bound $\delta = 2$}
For $\delta = 2$ it was sufficient to test uniformly random instances, as points of degree three were very common in the MSTs. The results are as follows.

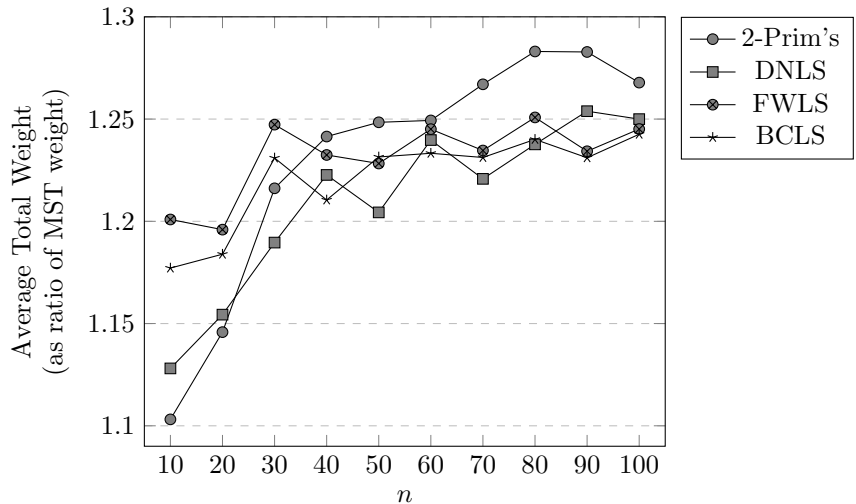
\begin{figure}[H]
\centering
 \scalebox{1}{
\begin{tikzpicture}
\begin{axis}[
    xlabel={$n$ },
    ylabel style={align=center},
    ylabel={Average Total Weight \\ (as ratio of MST weight)},
    xmin=5, xmax=105,
    ymin = 1.09, ymax = 1.3,
    xtick={10,20,30,40,50,60,70,80,90,100},
    legend pos= outer north east,
    ymajorgrids=true,
    grid style=dashed,
    cycle list name=black white,
]

   \addplot
    coordinates {
    (10,1.103136216)
    (20,1.145808295)
    (30,1.216136642)
    (40,1.241452574)
    (50,1.248416707)
    (60,1.2493357  )	    
    (70,1.266994615)
    (80,1.283054238)
    (90,1.282817008)
   (100,1.267828922)
    };
    \addlegendentry{2-Prim's}

   \addplot
  coordinates {
  (10,1.128063252)
  (20,1.154360907)
  (30,1.189599953)
  (40,1.222680411)
  (50,1.204361095)
  (60,1.239743521)	
  (70,1.220761715)
  (80,1.237631645)
  (90,1.253852736)
 (100,1.249976411)
   };
   \addlegendentry{DNLS}

    \addplot
    coordinates {
    (10,1.200895484)
    (20,1.195925228)
    (30,1.247275015)
    (40,1.232444038)
    (50,1.22826818 )
    (60,1.245071635)	    
    (70,1.234623611)
    (80,1.250804207)
    (90,1.234258973)
   (100,1.245117715)
    };
    \addlegendentry{FWLS}

    \addplot 
   coordinates {
    (10, 1.177135332)    
    (20, 1.183915574)
    (30, 1.230965217)
    (40, 1.210432022)
    (50, 1.231464681)
    (60, 1.233271088)
    (70, 1.231219389)
    (80, 1.240063827)
    (90, 1.231079416)
    (100,1.242581335)
    };
    \addlegendentry{BCLS}

\end{axis}
\end{tikzpicture} }
\caption{Best performing algorithms in terms of weight for $\delta =2$.}
\label{plot:2weightbest}
\end{figure}

From Figure~\ref{plot:2weightbest}, the DNLS, FWLS and BCLS algorithms were fairly close in terms of total weight with BCLS being the best performer for $n = 100$. BCLS mostly outperformed FWLS, but DNLS was at times the better performing algorithm for certain point values of $n$. After $n=30$, 2-Prim's was outperformed by the other three algorithms.

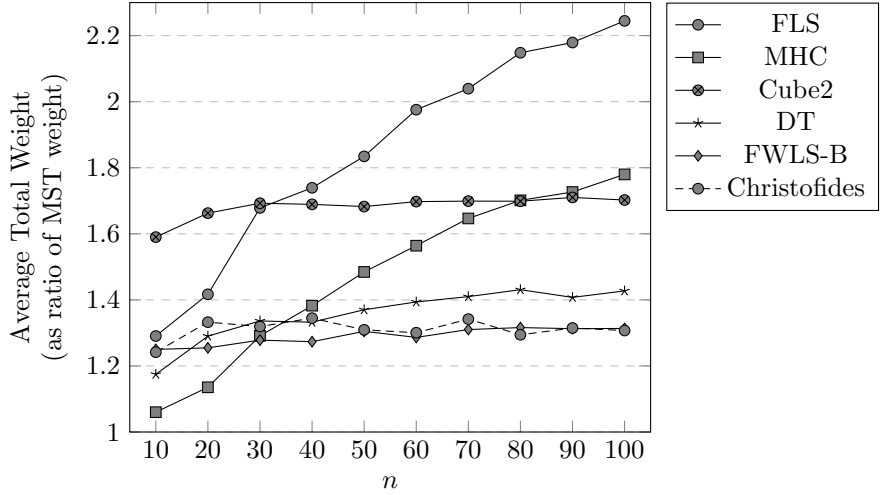
\begin{figure}[H]
\centering
 \scalebox{1}{
\begin{tikzpicture}
\begin{axis}[
    xlabel={$n$ },
    ylabel style={align=center},
    ylabel={Average Total Weight \\ (as ratio of MST weight)},
    xmin=5, xmax=105,
    ymin = 1, ymax = 2.3,
    xtick={10,20,30,40,50,60,70,80,90,100},
    legend pos= outer north east,
    ymajorgrids=true,
    grid style=dashed,
    cycle list name=black white,
]

  \addplot
  coordinates {
  (10,1.290545653)
  (20,1.416849406)
  (30,1.678409656)
  (40,1.739271318)
  (50,1.834458641)
  (60,1.975744749)	
  (70,2.039351806)
  (80,2.148152165)
  (90,2.17909747 )
 (100,2.244635689)
   };
   \addlegendentry{FLS}

  \addplot
  coordinates {
  (10,1.06002926)
  (20,1.13539880)
  (30,1.29101384)
  (40,1.38246171)
  (50,1.48453354)
  (60,1.56405460)	
  (70,1.64651746)
  (80,1.70133133)
  (90,1.72653761)
 (100,1.77990571)
   };
   \addlegendentry{MHC}

    \addplot
  coordinates {
  (10,1.590027516)
  (20,1.662623665)
  (30,1.69258591 )
  (40,1.689127846)
  (50,1.68253506 )
  (60,1.697463358)
  (70,1.69899917 )
  (80,1.698685496)
  (90,1.710100263)
 (100,1.70229584 )
   };
   \addlegendentry{Cube2}
   
   \addplot
    coordinates {
    (10,1.174651589)   
    (20,1.290246711)
    (30,1.336160772)
    (40,1.332468367)
    (50,1.370071284)
    (60,1.393658223)	    
    (70,1.40987065)
    (80,1.430815583)
    (90,1.407524807)
    (100,1.427191401)
    };
   \addlegendentry{DT}

      \addplot
    coordinates {
    (10,1.250057164)
    (20,1.254776677)
    (30,1.278164215)
    (40,1.273458468)
    (50,1.304747225)
    (60,1.286380565)	    
    (70,1.310177906)
    (80,1.31613192 )
    (90,1.312851651)
   (100,1.313508322)
    };
   \addlegendentry{FWLS-B} 
   
      \addplot
    coordinates {
    (10,1.241640247)
    (20,1.332572005)
    (30,1.31936429)
    (40,1.344700259)
    (50,1.309298221)
    (60,1.300637204)	    
    (70,1.341793853)
    (80,1.294376387)
    (90,1.314785928)
    (100,1.307226537)
    };
   \addlegendentry{Christofides}

\end{axis}
\end{tikzpicture} }
\caption{Worst performing algorithms in terms of weight for $\delta =2$.}
\label{plot:2weightworst}
\end{figure}

Figure~\ref{plot:2weightworst} graph seems to indicate that FWLS-B, DT, Christofides and Cube2 performed consistently with Christofides being the best performer in the group. The performances of FLS and MHC seemed to worsen as $n$ increased.

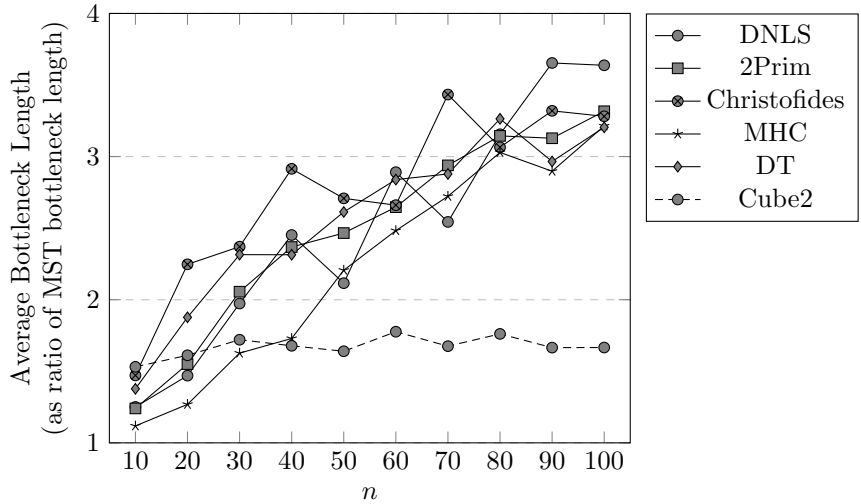
\begin{figure}[H]
\centering
 \scalebox{1}{
\begin{tikzpicture}
\begin{axis}[
    xlabel={$n$ },
    ylabel style={align=center},
    ylabel={Average Bottleneck Length \\ (as ratio of MST bottleneck length)},
    xmin=5, xmax=105,
    ymin = 1, ymax = 4,
    xtick={10,20,30,40,50,60,70,80,90,100},
    legend pos= outer north east,
    ymajorgrids=true,
    grid style=dashed,
    cycle list name=black white,
]

    \addplot
   coordinates {
  (10,1.250723057)
  (20,1.469482677)
  (30,1.973517316)
  (40,2.452015668)
  (50,2.115619313)
  (60,2.891084027)	
  (70,2.543552975)
  (80,3.155932298)
  (90,3.654059288)
 (100,3.637763487)
   };
   \addlegendentry{DNLS}

   \addplot
  coordinates {
  (10,1.240638647)
  (20,1.550225879)
  (30,2.055975523)
  (40,2.368177355)
  (50,2.46570765 )
  (60,2.646605558)	
  (70,2.938607423)
  (80,3.144538758)
  (90,3.128079503)
 (100,3.316617467)
   };
   \addlegendentry{2Prim}
   
      \addplot
  coordinates {
  (10,1.47120836)
  (20,2.247265618)
  (30,2.370524673)
  (40,2.914506954)
  (50,2.708029734)
  (60,2.660731658)	
  (70,3.433385712)
  (80,3.064905746)
  (90,3.319660966)
 (100,3.281106913)
   };
   \addlegendentry{Christofides}

   \addplot
   coordinates {
  (10,1.118243972)
  (20,1.268909434)
  (30,1.627005334)
  (40,1.729318829)
  (50,2.206153269)
  (60,2.483025954)	
  (70,2.722524642)
  (80,3.028869477)
  (90,2.897880614)
 (100,3.213801372)
   };
   \addlegendentry{MHC}

      \addplot
   coordinates {
  (10,1.377362878)
  (20,1.877185325)
  (30,2.314906788)
  (40,2.314156905)
  (50,2.613049351)
  (60,2.839172168)	
  (70,2.878672141)
  (80,3.263605046)
  (90,2.965792699)
 (100,3.204541688)
   };
   \addlegendentry{DT} 
           
     \addplot
   coordinates {
  (10,1.531207504)
  (20,1.611814041)
  (30,1.720427045)
  (40,1.678130833)
  (50,1.639985793)
  (60,1.77645911 )	
  (70,1.675516894)
  (80,1.761212759)
  (90,1.665369427)
 (100,1.665959224)
   };
   \addlegendentry{Cube2}

\end{axis}
\end{tikzpicture} }
\caption{Best performing algorithms in terms of bottleneck for $\delta =2$.}
\label{plot:2bottlebest}
\end{figure}

Figure~\ref{plot:2bottlebest} shows that Cube2 is the clear winner in terms of bottleneck length with a fairly consistent performance. The other algorithms seem to worsen in performance as $n$ increases, with DNLS and Christofides being the most variable in performance. The DT algorithm had the second best overall performance.

\begin{figure}[H]
\centering
 \scalebox{1}{
\begin{tikzpicture}
\begin{axis}[
    xlabel={$n$ },
    ylabel style={align=center},
    ylabel={Average Bottleneck Length \\ (as ratio of MST bottleneck length)},
    xmin=5, xmax=105,
    ymin = 1, ymax = 6.5,
    xtick={10,20,30,40,50,60,70,80,90,100},
    legend pos= outer north east,
    ymajorgrids=true,
    grid style=dashed,
    cycle list name=black white,
]
\addplot
   coordinates {
  (10,1.570034974)
  (20,2.449394184)
  (30,3.595880282)
  (40,3.867092923)
  (50,4.45120084 )
  (60,4.744881194)	
  (70,5.386882158)
  (80,6.05786724 )
  (90,6.058963797)
 (100,6.287427989)
   };
   \addlegendentry{FLS}

\addplot
   coordinates {
  (10,1.454946249)
  (20,1.928977569)
  (30,2.409056988)
  (40,2.693849405)
  (50,3.303807322)
  (60,3.055199074)	
  (70,3.498621377)
  (80,3.990921276)
  (90,3.793502831)
 (100,4.454779418)
   };
   \addlegendentry{FWLS-B}

	\addplot
   coordinates {
  (10, 1.429933772)
  (20, 1.729312097)
  (30, 2.618405494)
  (40, 2.516967476)
  (50, 3.091270985)
  (60, 3.408586501)	
  (70, 3.406421241)
  (80, 3.741566497)
  (90, 3.631529794)
  (100,4.124813249)
   };
   \addlegendentry{BCLS}

      \addplot
   coordinates {
  (10,1.426143625)
  (20,1.736138271)
  (30,2.6696669  )
  (40,2.674466563)
  (50,2.928816642)
  (60,3.599618729)	
  (70,3.418818329)
  (80,3.910220839)
  (90,3.691747709)
 (100,4.092553816)
   };
   \addlegendentry{FWLS}

\end{axis}
\end{tikzpicture} }
\caption{Worst performing algorithms in terms of bottleneck for $\delta =2$.}
\label{plot:2bottleworst}
\end{figure}
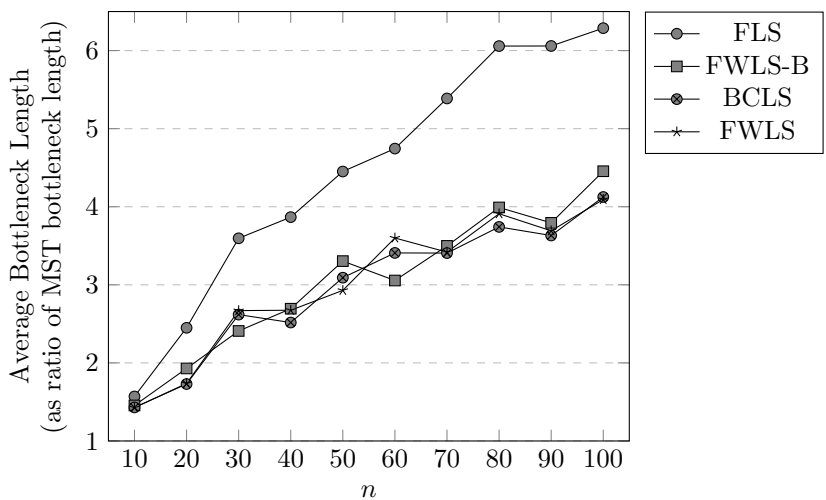

FLS was clearly the worst performing algorithm in Figure~\ref{plot:2bottleworst} with FWLS-B and BCLS being relatively close in performance in comparison.

\subsection{Degree bound $\delta = 3$: Uniform Instances}
For $\delta = 3$, we tested the algorithms on both the uniformly random instances and the special instances. In this section we present the results for the uniform instances with the results for the special instances in the next section. The uniform instances tested here were found by searching through a set of $100$ instances for each $n \in \{10,20, \dots, 100\}$ and selecting and testing only those instances whose MSTs contained points of degree 4 or 5. As such, the number of instances tested for each $n$ varied, with the number of instances increasing as $n$ increased. We give the number of instances tested for each $n$ in Table~\ref{tab:deg3tested}.

\begin{table}[]
\centering
\caption{Number of uniform instances tested for $\delta = 3$.}
\label{tab:deg3tested}
\begin{tabular}{|c|c|}
\hline
$n$ & Instances Tested \\ \hline
10  & 2                \\
20  & 17               \\
30  & 14               \\
40  & 20               \\
50  & 30               \\
60  & 32               \\
70  & 40               \\
80  & 50               \\
90  & 43               \\
100 & 59               \\ \hline
\end{tabular}
\end{table}

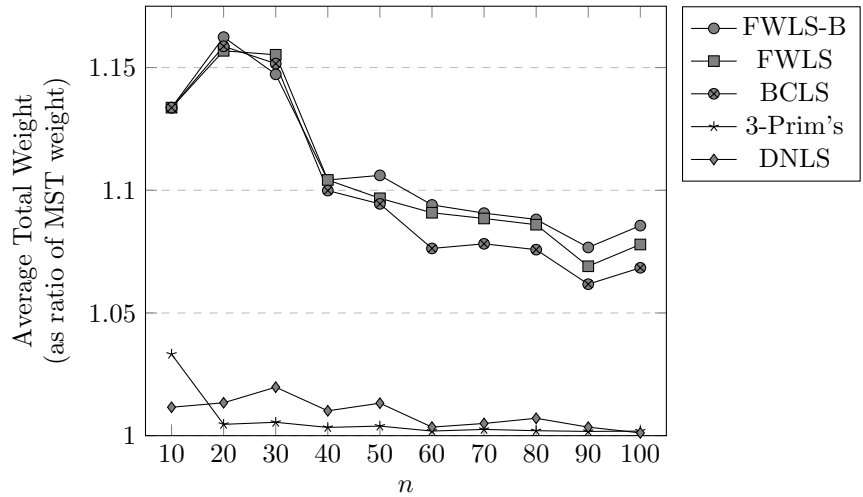
\begin{figure}[H]
\centering
 \scalebox{1}{
\begin{tikzpicture}
\begin{axis}[
    xlabel={$n$ },
    ylabel style={align=center},
    ylabel={Average Total Weight \\ (as ratio of MST weight)},
    xmin=5, xmax=105,
    ymin = 1.0, ymax = 1.175,
    xtick={10,20,30,40,50,60,70,80,90,100},
    legend pos= outer north east,
    ymajorgrids=true,
    grid style=dashed,
    cycle list name=black white,
]

\addplot
    coordinates {
    (10,1.133666163)(20,1.162420266)(30,1.147251766)(40,1.104192824)(50,1.106076681)(60,1.094033616)	    (70,1.090677495)(80,1.08811487)(90,1.076697782)(100,1.085596235)
    };
   \addlegendentry{FWLS-B}

\addplot
    coordinates {
    (10,1.133666163)(20,1.156835376)(30,1.155216049)(40,1.104192824)(50,1.096729247)(60,1.090837457)	    (70,1.088507985)(80,1.08597511)(90,1.069029623)(100,1.077903605)
    };
   \addlegendentry{FWLS}

\addplot
    coordinates {
    (10,1.133666163)(20,1.158672596)(30,1.151613236)(40,1.0998669)(50,1.094448293)(60,1.076294605)	    (70,1.078178698)(80,1.075800606)(90,1.061725816)(100,1.068401931)
    };
   \addlegendentry{BCLS}

\addplot
    coordinates {
    (10,1.033196935)(20,1.004613999)(30,1.005448773)(40,1.003356998)(50,1.003895446)(60,1.001871469)	    (70,1.002546848)(80,1.002017474)(90,1.001722011)(100,1.001858757)
    };
   \addlegendentry{3-Prim's}

\addplot
    coordinates {
    (10,1.011574013)(20,1.013391128)(30,1.01975616)(40,1.010118555)(50,1.013232913)(60,1.003482659)	    (70,1.004965557)(80,1.007085437)(90,1.003468788)(100,1.001116214)
    };
   \addlegendentry{DNLS}
   
\end{axis}
\end{tikzpicture} }
\caption{Best performing algorithms in terms of weight for $\delta =3$ (uniform instances).}
\label{plot:3weightbestU}
\end{figure}

From Figure~\ref{plot:3weightbestU}, we see that FWLS-B, FWLS and BCLS performed similarly, with BCLS slightly outperforming the other two. DNLS and 3-Prim's also had close performances with 3-Prim's slightly outperforming DNLS for most values of $n$ but not for $n = 100$. Both DNLS and 3-Prim's clearly outperformed FWLS-B, FWLS and BCLS.

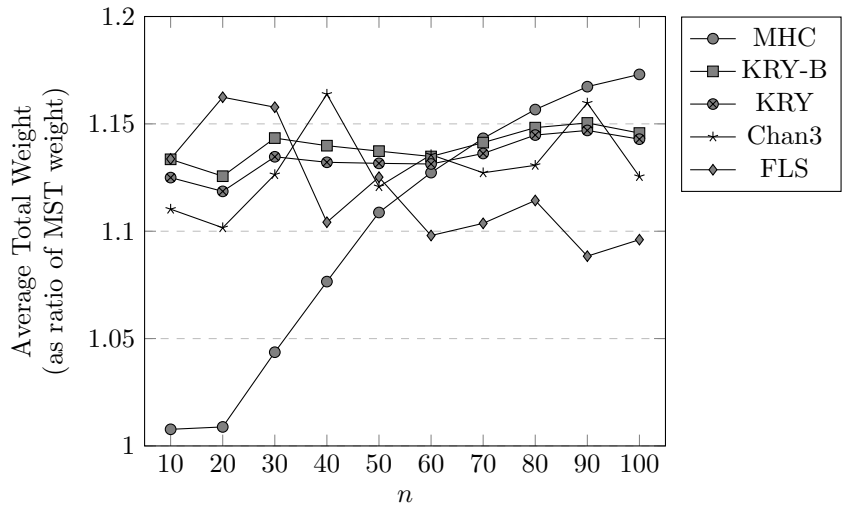
\begin{figure}[H]
\centering
 \scalebox{1}{
\begin{tikzpicture}
\begin{axis}[
    xlabel={$n$ },
    ylabel style={align=center},
    ylabel={Average Total Weight \\ (as ratio of MST weight)},
    xmin=5, xmax=105,
    ymin = 1.0, ymax = 1.2,
    xtick={10,20,30,40,50,60,70,80,90,100},
    legend pos= outer north east,
    ymajorgrids=true,
    grid style=dashed,
    cycle list name=black white,
]

\addplot
    coordinates {
    (10,1.007740287)(20,1.008850914)(30,1.043648819)(40,1.076537573)(50,1.108731871)(60,1.127290165)	    (70,1.143269558)(80,1.156662059)(90,1.167324473)(100,1.173030484)
    };
   \addlegendentry{MHC}

\addplot
    coordinates {
    (10,1.13349401)(20,1.125624511)(30,1.143379715)(40,1.139840024)(50,1.137305831)(60,1.134805891)	    (70,1.141278166)(80,1.148249352)(90,1.150426448)(100,1.145723424)
    };
   \addlegendentry{KRY-B}

\addplot
    coordinates {
    (10,1.124974618)(20,1.118559464)(30,1.134673828)(40,1.132105768)(50,1.131609118)(60,1.131313399)	    (70,1.136200779)(80,1.144781301)(90,1.146943163)(100,1.142867684)
    };
   \addlegendentry{KRY}
   
\addplot
    coordinates {
    (10,1.110268736)(20,1.101626264)(30,1.126369978)(40,1.163809584)(50,1.12076866)(60,1.135618803)	    (70,1.127215469)(80,1.130687713)(90,1.159687949)(100,1.125483054)
    };
   \addlegendentry{Chan3}

\addplot
    coordinates {
    (10,1.133666163)(20,1.162420266)(30,1.157759154)(40,1.104192824)(50,1.125182701)(60,1.097990244)	    (70,1.103669196)(80,1.11431801)(90,1.088383497)(100,1.096069573)
    };
   \addlegendentry{FLS}

\end{axis}
\end{tikzpicture} }
\caption{Worst performing algorithms in terms of weight for $\delta =3$ (uniform instances).}
\label{plot:3weightworstU}
\end{figure}

It can bee seen from Figure~\ref{plot:3weightworstU} that MHC performs well initially but performs increasingly worse as $n$ increases. The other algorithms perform reasonably consistently with no clear stand outs other than surprisingly FLS which performed the best in the group.  

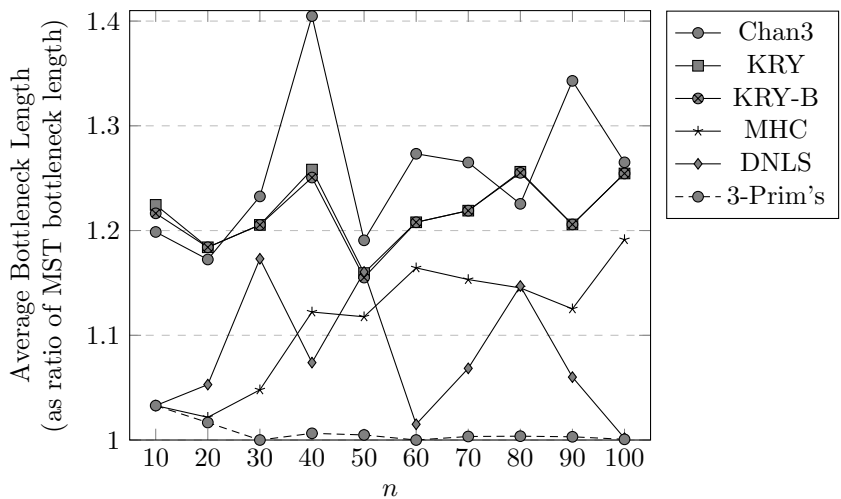
\begin{figure}[H]
\centering
 \scalebox{1}{
\begin{tikzpicture}
\begin{axis}[
    xlabel={$n$ },
    ylabel style={align=center},
    ylabel={Average Bottleneck Length \\ (as ratio of MST bottleneck length)},
    xmin=5, xmax=105,
    ymin = 1, ymax = 1.41,
    xtick={10,20,30,40,50,60,70,80,90,100},
    legend pos= outer north east,
    ymajorgrids=true,
    grid style=dashed,
    cycle list name=black white,
]

\addplot
    coordinates {
    (10,1.198657893)(20,1.172368855)(30,1.232581453)(40,1.404775468)(50,1.190656361)(60,1.273300027)(70,1.264980729)(80,1.225485834)(90,1.342905272)(100,1.265111663)
    };
   \addlegendentry{Chan3}

\addplot
    coordinates {
    (10,1.224502149)(20,1.184016768)(30,1.205385456)(40,1.258201615)(50,1.159402844)(60,1.207891052)	    (70,1.218982879)(80,1.256128714)(90,1.205821048)(100,1.25449931)
    };
   \addlegendentry{KRY}
   
\addplot
    coordinates {
    (10,1.216425439)(20,1.184016768)(30,1.205385456)(40,1.250609732)(50,1.155080821)(60,1.207891052)	    (70,1.218982879)(80,1.255099872)(90,1.205821048)(100,1.25449931)
    };
   \addlegendentry{KRY-B}

\addplot
    coordinates {
    (10,1.032868381)(20,1.021693293)(30,1.04789379)(40,1.122232542)(50,1.117655433)(60,1.164369228)	    (70,1.153210639)(80,1.145231163)(90,1.125231801)(100,1.191324024)
    };
   \addlegendentry{MHC}

\addplot
    coordinates {
    (10,1.032868381)(20,1.052909473)(30,1.173034848)(40,1.073890499)(50,1.160463044)(60,1.015021767)	    (70,1.068433877)(80,1.147053516)(90,1.060062059)(100,1.001748039)
    };
   \addlegendentry{DNLS}

\addplot
    coordinates {
    (10,1.032868381)(20,1.016695196)(30,1)(40,1.006387114)(50,1.004847176)(60,1)(70,1.003371165)(80,1.003661523)(90,1.00308493)(100,1.00065275)
    };
   \addlegendentry{3-Prim's}

\end{axis}
\end{tikzpicture} }
\caption{Best performing algorithms in terms of bottleneck for $\delta =3$ (uniform instances).}
\label{plot:3bottlebestU}
\end{figure}

The 3-Prim's algorithm is the clear winner for bottleneck in Figure~\ref{plot:3bottlebestU}, achieving an average value consistently close to that of the MST. DNLS also had a value close to the MST for $n = 100$, but with very variable performances for the other values of $n$. MHC was the third best performer with KRY and KRY-B each having very close performances.

\begin{figure}[H]
\centering
 \scalebox{1}{
\begin{tikzpicture}
\begin{axis}[
    xlabel={$n$ },
    ylabel style={align=center},
    ylabel={Average Bottleneck Length \\ (as ratio of MST bottleneck length)},
    xmin=5, xmax=105,
    ymin = 1.15, ymax = 3.75,
    xtick={10,20,30,40,50,60,70,80,90,100},
    legend pos= outer north east,
    ymajorgrids=true,
    grid style=dashed,
    cycle list name=black white,
]

\addplot
    coordinates {
    (10,1.367606377)(20,2.023293624)(30,2.41211837)(40,2.326619545)(50,2.601584525)(60,2.806823151)(70,3.069875381)(80,3.421197093)(90,2.933296474)(100,3.721612057)
    };
   \addlegendentry{FLS}

\addplot
    coordinates {
    (10,1.367606377)(20,1.992467433)(30,2.41211837)(40,2.326619545)(50,2.462820353)(60,2.751268429)(70,2.946966431)(80,3.222148265)(90,2.790156595)(100,3.540197311)
    };
   \addlegendentry{FWLS}

\addplot
    coordinates {
    (10,1.367606377)(20,2.023293624)(30,2.268206669)(40,2.326619545)(50,2.534049703)(60,2.806823151)(70,2.942646461)(80,3.195833073)(90,2.900086566)(100,3.457233186)
    };
   \addlegendentry{FWLS-B}

\addplot
    coordinates {
    (10,1.367606377)(20,2.01126031)(30,2.376640525)(40,2.29235461)(50,2.441860328)(60,2.415116943)(70,2.676869005)(80,2.962554635)(90,2.540796583)(100,3.194751479)
    };
   \addlegendentry{BCLS}

\end{axis}
\end{tikzpicture} }
\caption{Worst performing algorithms in terms of bottleneck for $\delta =3$ (uniform instances).}
\label{plot:3bottleworstU}
\end{figure}
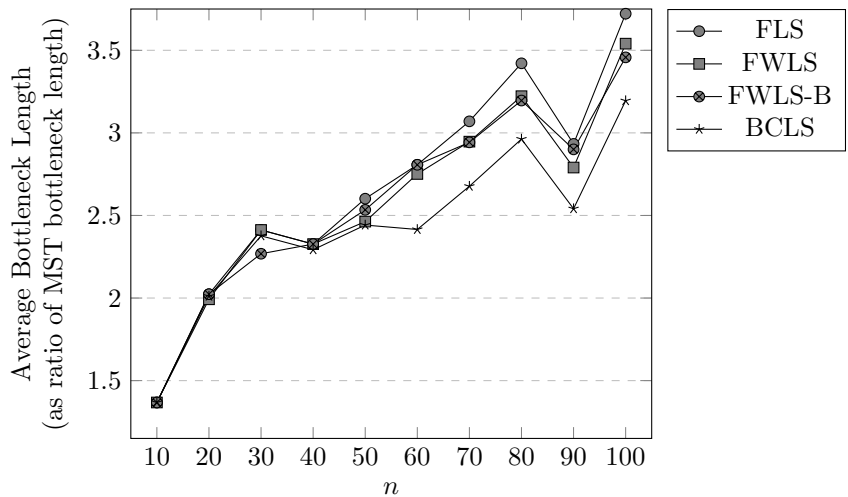

All algorithms in Figure~\ref{plot:3bottleworstU} perform similarly with BCLS being the best performer in the group.

\subsection{Degree bound $\delta = 3$: Special Instances}

\begin{figure}[H]
\centering
 \scalebox{1}{
\begin{tikzpicture}
\begin{axis}[
    xlabel={$n$ },
    ylabel style={align=center},
    ylabel={Average Total Weight \\ (as ratio of MST weight)},
    xmin=5, xmax=105,
    ymin = 1.0, ymax = 1.15,
    xtick={11,20,30,40,50,60,70,80,90,100},
    legend pos= outer north east,
    ymajorgrids=true,
    grid style=dashed,
    cycle list name=black white,
]

\addplot
    coordinates {
    (11,1.371743525)
    (20,1.163806221)
    (30,1.148054857)
    (40,1.146766749)
    (50,1.1252244  )
    (60,1.1237084  )	    
    (70,1.105669927)
    (80,1.090280619)
    (90,1.106975558)
   (100,1.09947391 )
    };
   \addlegendentry{FWLS}

   \addplot
    coordinates {
    (11,1.340134144)
    (20,1.14311223 )
    (30,1.135093514)
    (40,1.139422331)
    (50,1.112446271)
    (60,1.112318087)	    
    (70,1.094958677)
    (80,1.079622115)
    (90,1.091010332)
   (100,1.100046358)
    };
   \addlegendentry{BCLS}

\addplot
    coordinates {
    (11,1.04649731 )
    (20,1.084215826)
    (30,1.088843385)
    (40,1.088470457)
    (50,1.079299045)
    (60,1.088203285)	    
    (70,1.090524647)
    (80,1.090173509)
    (90,1.089649975)
   (100,1.096220352)
    };
   \addlegendentry{KRY}
   
\addplot
    coordinates {
    (11,1.04649731 )
    (20,1.083956353)
    (30,1.08785497 )
    (40,1.087103504)
    (50,1.080116405)
    (60,1.082403789)	    
    (70,1.088318045)
    (80,1.085586957)
    (90,1.083653188)
   (100,1.093852725)
    };
   \addlegendentry{Chan3}

\addplot
    coordinates {
    (11,1.010523511)
    (20,1.005756544)
    (30,1.007310761)
    (40,1.00507467 )
    (50,1.004904103)
    (60,1.005542748)	    
    (70,1.004347838)
    (80,1.005023718)
    (90,1.005830473)
   (100,1.005816999)
    };
   \addlegendentry{3-Prim's}

\addplot
    coordinates {
    (11,1.04465486 )
    (20,1.012455978)
    (30,1.009866474)
    (40,1.003246888)
    (50,1.003444933)
    (60,1.007416822)	    
    (70,1.004298462)
    (80,1.008635373)
    (90,1.002467316)
   (100,1.002681619)
    };
   \addlegendentry{DNLS}

\end{axis}
\end{tikzpicture} }
\caption{Best performing algorithms in terms of weight for $\delta =3$ (special instances).}
\label{plot:3weightbestS}
\end{figure}
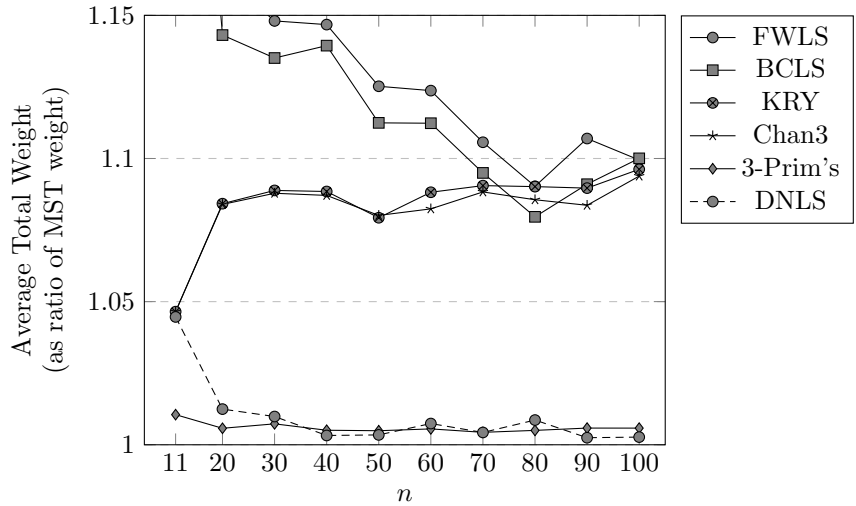

DNLS and 3-Prim's are the clear winners again for the special instances in Figure~ \ref{plot:3weightbestS}, with DNLS performing slightly better in the end. The Chan3 and and KRY algorithms had fairly consistent performances with the performances of BCLS and FWLS improving as $n$ increased. 

\begin{figure}[H]
\centering
 \scalebox{1}{
\begin{tikzpicture}
\begin{axis}[
    xlabel={$n$ },
    ylabel style={align=center},
    ylabel={Average Total Weight \\ (as ratio of MST weight)},
    xmin=5, xmax=105,
    ymin = 1, ymax =3,
    xtick={11,20,30,40,50,60,70,80,90,100},
    legend pos= outer north east,
    ymajorgrids=true,
    grid style=dashed,
    cycle list name=black white,
]

\addplot
    coordinates {
    (11,2.079468173)
    (20,1.743552849)
    (30,1.716426365)
    (40,2.169106484)
    (50,2.106835128)
    (60,2.542318576)	    
    (70,2.529462879)
    (80,2.657653525)
    (90,2.716034989)
   (100,2.971708413)
    };
   \addlegendentry{FLS}

\addplot
    coordinates {
    (11,1.610202299)
    (20,1.44676837 )
    (30,1.359735588)
    (40,1.537765143)
    (50,1.426769851)
    (60,1.522604622)	    
    (70,1.462874116)
    (80,1.541124698)
    (90,1.485444172)
   (100,1.541890823)
    };
   \addlegendentry{FWLS-B}
   
   \addplot
    coordinates {
    (11,1.01397532 )
    (20,1.003594649)
    (30,1.014490407)
    (40,1.025727878)
    (50,1.048783391)
    (60,1.058678802)	    
    (70,1.085991008)
    (80,1.085502112)
    (90,1.101862526)
   (100,1.124513022)
    };
   \addlegendentry{MHC}

\addplot
    coordinates {
    (11,1.06574725 )
    (20,1.101468654)
    (30,1.101006785)
    (40,1.101026363)
    (50,1.091120568)
    (60,1.10400857 )	    
    (70,1.105656588)
    (80,1.102476424)
    (90,1.103175205)
   (100,1.119488109)
    };
   \addlegendentry{KRY-B}

\end{axis}
\end{tikzpicture} }
\caption{Worst performing algorithms in terms of weight for $\delta =3$ (special instances).}
\label{plot:3weightworstS}
\end{figure}
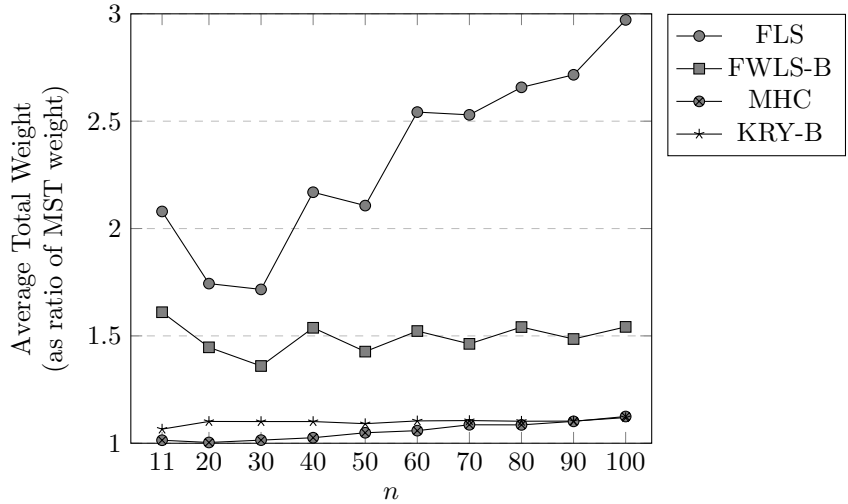

The KRY-B and MHC algorithms had fairly close performances in Figure~\ref{plot:3weightworstS} and FWLS-B performed reasonably consistently.

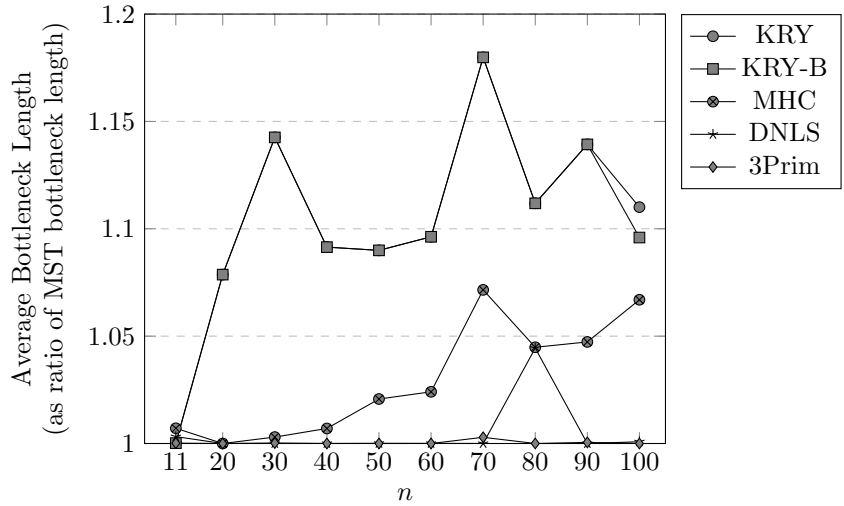
\begin{figure}[H]
\centering
 \scalebox{1}{
\begin{tikzpicture}
\begin{axis}[
    xlabel={$n$ },
    ylabel style={align=center},
    ylabel={Average Bottleneck Length \\ (as ratio of MST bottleneck length)},
    xmin=5, xmax=105,
    ymin = 1.0, ymax = 1.2,
    xtick={11,20,30,40,50,60,70,80,90,100},
    legend pos= outer north east,
    ymajorgrids=true,
    grid style=dashed,
    cycle list name=black white,
]

\addplot
    coordinates {
    (11,1.000172983)
    (20,1.078662067)
    (30,1.14257574 )
    (40,1.091461908)
    (50,1.089952798)
    (60,1.096274277)
    (70,1.179848446)
    (80,1.111860913)
    (90,1.139303836)
   (100,1.110076698)
    };
   \addlegendentry{KRY}

\addplot
    coordinates {
    (11,1.000172983)
    (20,1.078662067)
    (30,1.14257574 )
    (40,1.091461908)
    (50,1.089952798)
    (60,1.096274277)
    (70,1.179848446)
    (80,1.111860913)
    (90,1.139303836)
   (100,1.095929699)
    };
   \addlegendentry{KRY-B}

\addplot
    coordinates {
    (11,1.007088868)
    (20,1          )
    (30,1.00299782 )
    (40,1.006996827)
    (50,1.020717779)
    (60,1.024072258)
    (70,1.07154805 )
    (80,1.044824068)
    (90,1.047301237)
   (100,1.066942399)
    };
   \addlegendentry{MHC}

\addplot
    coordinates {
    (11,1.003205755)
    (20,1          )
    (30,1.000033575)
    (40,1.000022255)
    (50,1.000087181)
    (60,1.000047875)
    (70,1.000026067)
    (80,1.044410819)
    (90,1          )
   (100,1.000813439)
    };
   \addlegendentry{DNLS}

\addplot
    coordinates {
    (11,1.000172983)
    (20,1          )
    (30,1.000200501)
    (40,1          )
    (50,1          )
    (60,1          )
    (70,1.002888971)
    (80,1          )
    (90,1.000521998)
   (100,1          )
    };
   \addlegendentry{3Prim}

\end{axis}
\end{tikzpicture} }
\caption{Best performing algorithms in terms of bottleneck for $\delta =3$ (special instances).}
\label{plot:3bottlebestS}
\end{figure}
DNLS and 3-Prim's once again outperformed the others in terms of bottleneck as seen Figure~\ref{plot:3bottlebestS} with DNLS again being the variable of the two algorithms. MHC was the third best performer with KRY and KRY-B once again performing similarly for bottleneck.

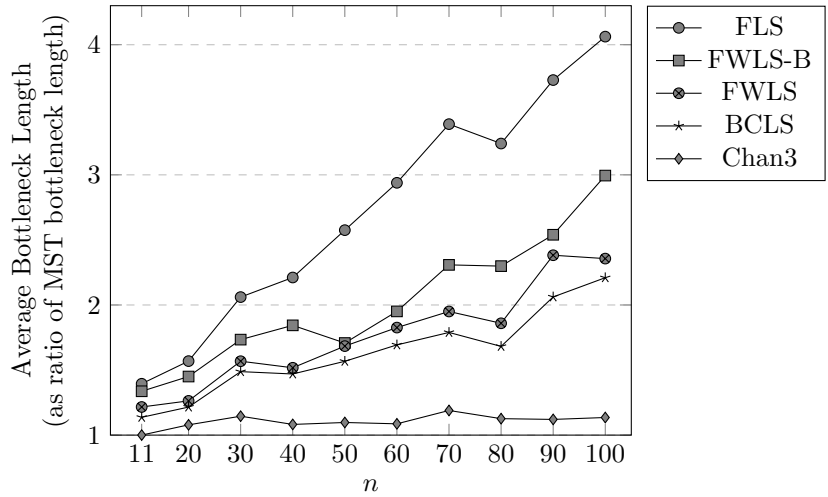
\begin{figure}[H]
\centering
 \scalebox{1}{
\begin{tikzpicture}
\begin{axis}[
    xlabel={$n$ },
    ylabel style={align=center},
    ylabel={Average Bottleneck Length \\ (as ratio of MST bottleneck length)},
    xmin=5, xmax=105,
    ymin = 1.0, ymax = 4.3,
    xtick={11,20,30,40,50,60,70,80,90,100},
    legend pos= outer north east,
    ymajorgrids=true,
    grid style=dashed,
    cycle list name=black white,
]

\addplot
    coordinates {
    (11,1.394670552)
    (20,1.568640649)
    (30,2.060637895)
    (40,2.211186699)
    (50,2.574989386)
    (60,2.938833372)
    (70,3.389726952)
    (80,3.240626359)
    (90,3.728420848)
   (100,4.062201581)
    };
   \addlegendentry{FLS}

\addplot
    coordinates {
    (11,1.336909979)
    (20,1.450174501)
    (30,1.733787632)
    (40,1.843535746)
    (50,1.707678709)
    (60,1.950771401)
    (70,2.308217135)
    (80,2.29866636 )
    (90,2.539588608)
   (100,2.994441896)
    };
   \addlegendentry{FWLS-B}

\addplot
    coordinates {
    (11,1.216139958)
    (20,1.263137542)
    (30,1.567856137)
    (40,1.517416595)
    (50,1.683869649)
    (60,1.826287031)
    (70,1.949976432)
    (80,1.85934523 )
    (90,2.382281066)
   (100,2.356436937)
    };
   \addlegendentry{FWLS}

\addplot
    coordinates {
    (11,1.136645405)
    (20,1.215436172)
    (30,1.487803939)
    (40,1.469689043)
    (50,1.567219255)
    (60,1.693182575)
    (70,1.788950266)
    (80,1.681718886)
    (90,2.061886402)
   (100,2.209515517)
    };
   \addlegendentry{BCLS}

\addplot
    coordinates {
    (11,1.000172983)
    (20,1.078888285)
    (30,1.145326072)
    (40,1.082072298)
    (50,1.096990973)
    (60,1.086730565)
    (70,1.189338483)
    (80,1.126450247)
    (90,1.120592472)
   (100,1.135301733)
    };
   \addlegendentry{Chan3}

\end{axis}
\end{tikzpicture} }
\caption{Worst performing algorithms in terms of bottleneck for $\delta =3$ (special instances).}
\label{plot:3bottleworstS}
\end{figure}

The Chan3 algorithm was the clear best performer in Figure~\ref{plot:3bottleworstS} with a fairly consistent performance whilst the other algorithms seemed to worsen in performance as $n$ increased.

\subsection{Degree bound $\delta = 4$}
Since points of degree $5$ or more are very rare in MST's of uniformly random point sets in the plane, for $\delta = 4$ , we only tested our special instances. The results are shown below.

\begin{figure}[H]
\centering
 \scalebox{1}{
\begin{tikzpicture}
\begin{axis}[
    xlabel={$n$ },
    ylabel style={align=center},
    ylabel={Average Total Weight \\ (as ratio of MST weight)},
    xmin=5, xmax=105,
    ymin = 1.0, ymax = 1.005,
    xtick={11,20,30,40,50,60,70,80,90,100},
    ytick={1,1.001,1.002,1.003,1.004,1.005},
    yticklabels={1,1.001,1.002,1.003,1.004,1.005},
    legend pos= outer north east,
    ymajorgrids=true,
    grid style=dashed,
    cycle list name=black white,
]

\addplot
    coordinates {
    (11,1.003320981)
    (20,1.00171533 )
    (30,1.001127108)
    (40,1.001162337)
    (50,1.001151991)
    (60,1.000810593)	    
    (70,1.000771597)
    (80,1.001060721)
    (90,1.000800901)
   (100,1.001043638)
    };
   \addlegendentry{4-Prim's}

\addplot
    coordinates {
    (11,1.001724857)
    (20,1.000850476)
    (30,1.000603451)
    (40,1.000632267)
    (50,1.000593298)
    (60,1.000565073)	    
    (70,1.000748352)
    (80,1.000651762)
    (90,1.000460966)
   (100,1.000532953)
    };
   \addlegendentry{Chan4}

\addplot
    coordinates {
    (11,1.023085848)
    (20,1.001337108)
    (30,1.000492658)
    (40,1.000348202)
    (50,1.00024502 )
    (60,1.000248151)	    
    (70,1.000209202)
    (80,1.000197296)
    (90,1.000170412)
   (100,1.000226339)
    };
   \addlegendentry{DNLS}

\end{axis}
\end{tikzpicture} }
\caption{Best performing algorithms in terms of weight for $\delta =4$.}
\label{plot:4weightbest}
\end{figure}
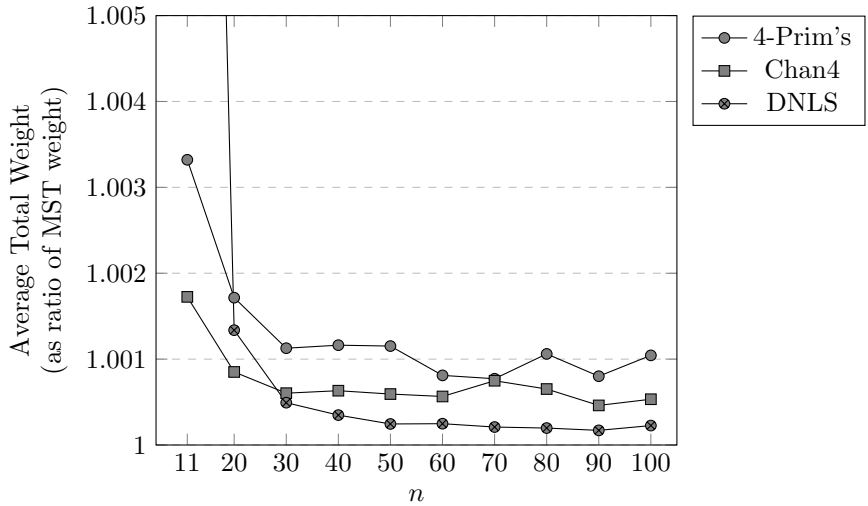

DNLS was the winner in Figure~\ref{plot:4weightbest}, followed by Chan4 (after $n = 70$) and 4-Prim's, although all algorithms gave trees of similar weights to MSTs.

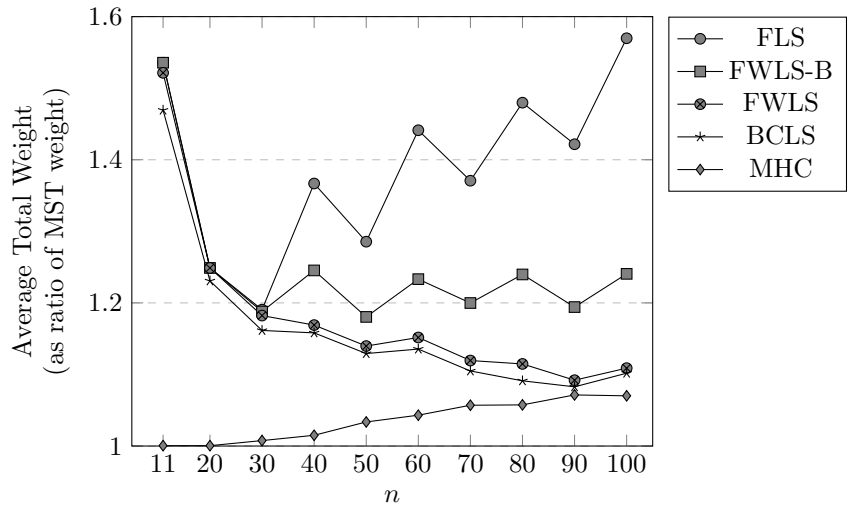
\begin{figure}[H]
\centering
 \scalebox{1}{
\begin{tikzpicture}
\begin{axis}[
    xlabel={$n$ },
    ylabel style={align=center},
    ylabel={Average Total Weight \\ (as ratio of MST weight)},
    xmin=5, xmax=105,
    ymin = 1.0, ymax = 1.6,
    xtick={11,20,30,40,50,60,70,80,90,100},
    legend pos= outer north east,
    ymajorgrids=true,
    grid style=dashed,
    cycle list name=black white,
]

\addplot
    coordinates {
    (11,1.535607184)
    (20,1.248826476)
    (30,1.190572567)
    (40,1.366826328)
    (50,1.285540226)
    (60,1.441242823)	    
    (70,1.370725699)
    (80,1.479761614)
    (90,1.421670967)
   (100,1.569777856)
    };
   \addlegendentry{FLS}
   
\addplot
    coordinates {
    (11,1.535607184)
    (20,1.248826476)
    (30,1.187928679)
    (40,1.245384328)
    (50,1.180407154)
    (60,1.233191475)	    
    (70,1.1998432  )
    (80,1.239749176)
    (90,1.19419264 )
   (100,1.240587058)
    };
   \addlegendentry{FWLS-B}

\addplot
    coordinates {
    (11,1.521442279)
    (20,1.248826476)
    (30,1.182283379)
    (40,1.169002512)
    (50,1.139674804)
    (60,1.151673802)	    
    (70,1.119468032)
    (80,1.114722248)
    (90,1.091995578)
   (100,1.108820555)
    };
   \addlegendentry{FWLS}

\addplot
    coordinates {
    (11,1.469297116)
    (20,1.230726404)
    (30,1.161498754)
    (40,1.158107722)
    (50,1.129246303)
    (60,1.135463353)	    
    (70,1.104893416)
    (80,1.091123284)
    (90,1.082732487)
   (100,1.101758555)
    };
   \addlegendentry{BCLS}

\addplot
    coordinates {
    (11,1.000506699)
    (20,1.000513129)
    (30,1.007613061)
    (40,1.015006856)
    (50,1.033640512)
    (60,1.043104037)	    
    (70,1.05699349 )
    (80,1.057492078)
    (90,1.071402302)
   (100,1.070163458)
    };
   \addlegendentry{MHC}

\end{axis}
\end{tikzpicture} }
\caption{Worst performing algorithms in terms of weight for $\delta =4$.}
\label{plot:4weightworst}
\end{figure}

MHC was the winner of Figure~\ref{plot:4weightworst}, with FWLS and BCLS having similar performances to each other. 

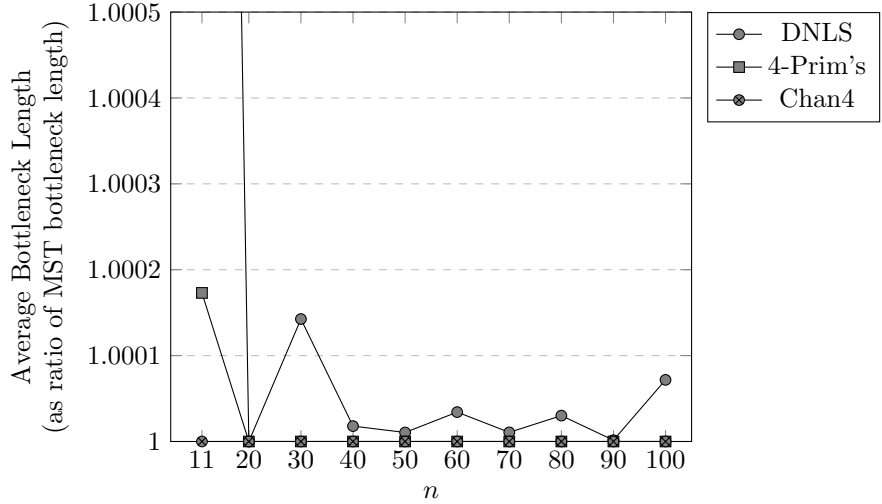
\begin{figure}[H]
\centering
 \scalebox{1}{
\begin{tikzpicture}
\begin{axis}[
    xlabel={$n$ },
    ylabel style={align=center},
    ylabel={Average Bottleneck Length \\ (as ratio of MST bottleneck length)},
    xmin=5, xmax=105,
    ymin = 1.0, ymax = 1.0005,
    xtick={11,20,30,40,50,60,70,80,90,100},
    ytick={1,1.0001,1.0002,1.0003,1.0004,1.0005},
    yticklabels={1,1.0001,1.0002,1.0003,1.0004,1.0005},
    legend pos= outer north east,
    ymajorgrids=true,
    grid style=dashed,
    cycle list name=black white,
]

\addplot
    coordinates {
    (11,1.003183683)
    (20,1          )
    (30,1.000142577)
    (40,1.000017901)
    (50,1.000010496)
    (60,1.000034212)
    (70,1.000010629)
    (80,1.000030118)
    (90,1.000001515)
   (100,1.000071757)
    };
   \addlegendentry{DNLS} 

\addplot
    coordinates {
    (11,1.000172983)(20,1)(30,1)(40,1)(50,1)(60,1)(70,1)(80,1)(90,1)(100,1)
    };
   \addlegendentry{4-Prim's}

\addplot
    coordinates {
    (11,1)(20,1)(30,1)(40,1)(50,1)(60,1)(70,1)(80,1)(90,1)(100,1)
    };
   \addlegendentry{Chan4}

\end{axis}
\end{tikzpicture} }
\caption{Best performing algorithms in terms of bottleneck for $\delta =4$.}
\label{plot:4bottlebest}
\end{figure}

All algorithms in Figure~\ref{plot:4bottlebest} had performances exceptionally close on average to the bottleneck values of the MSTs, with the Chan4 algorithm achieving the bottleneck value of the MST in all instances (clearly the edges incident to the nodes of degree 5 were not bottleneck edges).

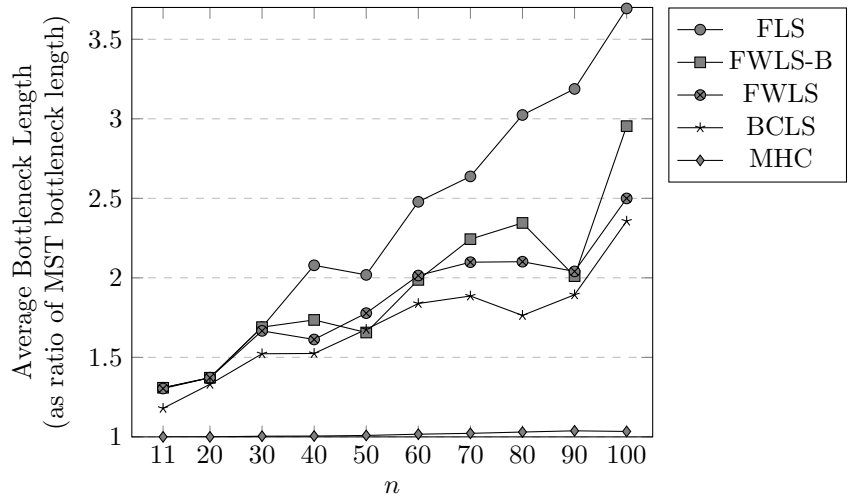
\begin{figure}[H]
\centering
 \scalebox{1}{
\begin{tikzpicture}
\begin{axis}[
    xlabel={$n$ },
    ylabel style={align=center},
    ylabel={Average Bottleneck Length \\ (as ratio of MST bottleneck length)},
    xmin=5, xmax=105,
    ymin = 1.0, ymax = 3.7,
    xtick={11,20,30,40,50,60,70,80,90,100},
    legend pos= outer north east,
    ymajorgrids=true,
    grid style=dashed,
    cycle list name=black white,
]

\addplot
    coordinates {
    (11,1.308856626)
    (20,1.372018776)
    (30,1.689983773)
    (40,2.078766002)
    (50,2.018516301)
    (60,2.478154331)
    (70,2.637596856)
    (80,3.02389912 )
    (90,3.187837414)
   (100,3.693670624)
    };
   \addlegendentry{FLS}

\addplot
    coordinates {
    (11,1.308856626)
    (20,1.372018776)
    (30,1.689983773)
    (40,1.735266396)
    (50,1.656188751)
    (60,1.987119438)
    (70,2.24315108 )
    (80,2.345166817)
    (90,2.010738426)
   (100,2.953804136)
    };
   \addlegendentry{FWLS-B}

\addplot
    coordinates {
    (11,1.304021473)
    (20,1.372018776)
    (30,1.666501516)
    (40,1.612461824)
    (50,1.777908918)
    (60,2.014449192)
    (70,2.098396711)
    (80,2.101378339)
    (90,2.040476624)
   (100,2.499334217)
    };
   \addlegendentry{FWLS}

\addplot
    coordinates {
    (11,1.179592106)
    (20,1.3307983  )
    (30,1.522610247)
    (40,1.524436839)
    (50,1.675901273)
    (60,1.83863058 )
    (70,1.885771842)
    (80,1.763677267)
    (90,1.893507127)
   (100,2.356647768)
    };
   \addlegendentry{BCLS}
   
   \addplot
    coordinates {
    (11,1          )
    (20,1          )
    (30,1.004369657)
    (40,1.005307491)
    (50,1.008989593)
    (60,1.016923957)
    (70,1.022259157)
    (80,1.030474513)
    (90,1.03816478 )
   (100,1.034035118)
    };
   \addlegendentry{MHC}

\end{axis}
\end{tikzpicture} }
\caption{Worst performing algorithms in terms of bottleneck for $\delta =4$.}
\label{plot:4bottleworst}
\end{figure}

From Figure~\ref{plot:4bottleworst}, it can be seen that the MHC algorithm perfomed consistently well, whereas the performances of the other algorithms seemed to worsen as $n$ increased.

\section{Results Discussion}  
We will give a brief summary of results of each algorithm separately.

\subsection{DT}
For the total weight objective, the DT algorithm fairly average in its performance compared to the other algorithms tested for $\delta = 2$. It was however, the second best algorithm in terms of bottleneck, beaten only by the Cube2 approximation algorithm which has performance guarantees for the bottleneck objective. The DT algorithm's success in the bottleneck objective compared to other algorithms could be due to the edge duplicating step not increasing the length of the bottleneck edge. It is only during the short-cutting process in which the bottleneck value can increase.

\subsection{Christofides}
Christofides algorithm performed similarly to FWLS-B for the min-sum objective, although both algorithms had a fairly average performance in comparison to the more successful algorithms. As expected, Christofides did outperform DT, although not by much. In terms of bottleneck objective, Christofides had a similar performance for $n=100$ to 2Prim, MHC and DT, although the performance of Chirstofides seemed to vary a lot more than these other algorithms.

\subsection{Cube2}
For the bottleneck objective when $\delta = 2$, this algorithm clearly outperformed all other algorithms tested and was fairly consistent in performance. However, it was of the worst performing algorithms for total weight, which is not hugely uprising as it was mostly intended as bottleneck algorithm. Because of its performance with the bottleneck objective and its time efficiency (see Table~\ref{tab:2MST_tim2}), this makes it the most suitable for solving the Euclidean bottleneck travelling salesman path problem.

\subsection{KRY and KRY-B}
Both these algorithm seemed to perform better relative to other algorithms when given the special instances rather than uniform instances. This is likely due to the fact that the local search algorithms can perform a small number of iterations when the feasibility error of an instance is low, whereas the recursive local edge swap algorithms must iterate over the entire tree. As would be expected, KRY-B outperformed KRY in bottleneck value (although not by much), whereas KRY outperformed KRY-B for total weight. KRY was outperformed by Chan3 in terms of total weight but the opposite seemed to be the case for the bottleneck length. However, both KRY and KRY-B were clearly outperformed by the DNLS and 3-Prim's algorithms in both criteria.

\subsection{Chan3}
Like KRY and KRY-B, this algorithm performed better on the special instances than uniform instances. It outperformed the KRY algorithms in total weight but not bottleneck, however it was beaten in both criteria by DNLS and 3-Prim's.

\subsection{Chan4}
This algorithm performed exceptionally well in terms of bottleneck value on the special instances for $\delta =4$, achieving the MST's value in all cases. It was also the second best performer for total weight, beaten only by the significantly more computationally expensive DNLS algorithm.

\subsection{FLS}
With the surprising exception of average total weight of uniform instances for $\delta = 3$, this algorithm was the clear worst performer in both criteria out of all algorithms tested. This is to be expected since it is a search algorithm which only takes feasibility error into consideration, not edge lengths of swapped edges. As such, it was meant to be used merely as a benchmark to which the other algorithms are compared.

\subsection{FWLS and FWLS-B}
Even though FWLS-B was intended as the bottleneck version of FWLS, it only outperformed its counterpart in either criteria for the uniform instances when $\delta = 3$ in which it only slightly outperformed FWLS in bottleneck length. FWLS had a similar performance in both criteria to BCLS although was usually outperformed, with the exception being for bottleneck for $\delta = 2$ in which FWLS was still outperformed by four other algorithms. With the exception of its performance with respect total weight for $\delta = 2$, FWLS was of the least effective algorithms for both criteria.

\subsection{Bi-Criteria Local Search}
As mentioned previously, BCLS performed similarly to FWLS in both objective criteria, although BCLS tended to do slightly better than FWLS. For $\delta = 3$ and $4$, BCLS was always outperformed by DNLS and $\delta$-Prim's, however it was the best performing algorithm in terms of total weight for $\delta = 2$.

\subsection{Diminishing Neighbourhood Local Search}
With the exception of the total weight objective for $\delta = 2$, in which it was still competitive, DNLS was clearly the best performing of the local search algorithms. It was the best performing algorithm (for $n = 100$) in terms of total weight for $\delta =3$ (both sets of instances) and $\delta = 4$, and the second best performing algorithm for these degree bounds in terms of bottleneck length. Its main competitor for performance was the $\delta$-Prim's algorithm which it performed similarly to. One factor that sets DNLS apart from $\delta$-Prim's however is running time (see tables \ref{tab:2MST_tim1}, \ref{tab:3MSTU_tim1}, \ref{tab:3MSTS_tim1}, \ref{tab:2MST_tim1}) as DNLS had the longest running time of all algorithms, whereas $\delta$-Prim's was of the best (see tables \ref{tab:2MST_tim2}, \ref{tab:3MSTU_tim2}, \ref{tab:3MSTS_tim2}, \ref{tab:2MST_tim2}). Another distinction is that while both algorithms were overall fairly consistent their performance for the various values of $n$, DNLS tended to be a slightly more variable in performance than $\delta$-Prim's.

\subsection{$\delta$-Prim's}
The $\delta$-Prim's algorithm was one of the better performing algorithms overall for both criteria and the best performing algorithm (for $n = 100$) in terms of bottleneck for $\delta =3$ (both sets of instances) and second best performing for $\delta = 4$. In addition, it was second only to the similarly performing DNLS algorithm in terms of total weight for $\delta = 3$, and third behind Chan4 and DNLS for $\delta = 4$. It was also one of the more time efficient algorithms among those tested (see tables \ref{tab:2MST_tim2}, \ref{tab:3MSTU_tim2}, \ref{tab:3MSTS_tim2}, \ref{tab:2MST_tim2}).

\subsection{MHC}
For $\delta =2$, the MHC algorithm was among the worst performing algorithms for total weight, but was very close to being the second best algorithm for bottleneck length (for $n = 100$). For $\delta = 3$, the algorithm was again outperformed by others in terms of total weight, but was third overall for bottleneck length for $n = 100$ behind the similarly performing DNLS and 3-Prim's algorithms. For $\delta =4$, the algorithm was fourth overall for $n = 100$ in terms of both total weight and bottleneck length behind the 4-Prim's, Chan4, and DNLS algorithms. Overall, the algorithm seemed to perform relatively better for the bottleneck objective compared to the total weight objective and its performance in terms of total weight seemed to worsen as $n$ increased.

\section{Conclusion}
There are many conclusions that one can infer from the results presented in this paper, some of which may motivate further investigation. One such conclusion is that 2-MST problem has a significantly different structure to the 3-MST and 4-MST problems in the Euclidean plane (and similarly for the bottleneck versions). This is evidenced by the fact that the algorithms that were the best performing when $\delta = 3$ or $4$ did not achieve the same success when $\delta =2$. For instance, DNLS and $\delta$-Prim's consistently outperformed the other algorithms in terms of weight for $\delta = 3$ and $4$, but were beaten by both BCLS and FWLS when $n = 100$ for $\delta = 2$. Typically, the BCLS and FWLS local search algorithms were not among the better performing algorithms for $\delta =3 $ and $4$, and yet appear to be of the betters options for accuracy in terms of total weight for $\delta =2$. Also, whereas DNLS and $\delta$-Prim's were the best performing algorithms for $\delta = 3$ and beaten only by the similarly performing Chan4 algorithm for $\delta = 4$, both DNLS and 2-Prim's were outperformed in terms of bottleneck length for $\delta = 2$ by MHC, DT and the Cube2 approximation algorithm, none of which are edge swap algorithms. These results seem to indicate that the algorithmic approach to solving the problem when $\delta = 2$ should use different methodologies then the approach to solving the general Euclidean bounded degree spanning tree problem. Furthermore, since it can be seen that there is no algorithm that came close to outperforming all others in both objective criteria for $\delta =2$, the results would indicate that the approach to solving the bottleneck version of the problem should also be different to that of the conventional total weight version for this degree bound.\\
In terms of the total weight objective for $\delta = 2$, no algorithm in particular seems to stand out among from the others, however the success of BCLS and FWLS suggest that an edge swap local search approach may be an appropriate methodology for solving the problem. However, for the $2$-MBST problem, the Cube2 algorithm quite clearly outperformed all others. Due to it also being quite time efficient (see Table~\ref{tab:2MST_tim2}), the Cube2 algorithm appears to be suitable choice for practical use for the $2$-MBST problem. Since the algorithm is relatively simple, it is possible that a more sophisticated algorithm based upon similar principles could be developed in the future.\\
Another observation that can be made is the difference in performance of the approximation algorithms for $\delta = 3$ (KRY, KRY-B, and Chan3) when applied to the set of uniform instances compared to the set of special instances. Whilst the approximations were among the worst performing in terms of total weight for the uniform instances (even being worse than FLS), they consistently managed to outperform the BCLS and FWLS algorithms for the special instances. This is mostly likely due to the fact that the approximation algorithms recursively apply edge swaps to the entire MST rather than just the initial nodes of high degree. Since the uniform instances have relatively small number of nodes of degree 4, this may be an inefficient strategy compared to the local search algorithms which only perform edge swaps that decrease feasibility error and thus would involve fewer edge swaps overall. Since the special instances have many nodes of degree 4 and 5, these local edge swap algorithms lose their advantage over the approximation algorithms, which is reflected by the results.\\
Another observation is that the $\delta$-Prim's algorithm was one of the better performing algorithms for $\delta = 3$ and $4$, being the best performing algorithm in terms of bottleneck for $\delta = 3$ when $n = 100$. This results seems almost counter-intuitive as the $\delta$-Prim's algorithm is merely a naive modification of Prim's algorithm for the MST and lacks the sophistication of some of the other approaches. This could be a consequence of the Euclidean MST sharing a large number of edges with an optimal $3$-MST/MBST or $4$-MST/MBST and may warrant further investigation.\\
The Chan4 algorithm happened to be the best performing algorithm for the $4$-MBST problem. Whilst this may have some relation to the relatively small approximation ratio of the algorithm, this result was also most likely contributed to by the fact that the special instances for $\delta = 4$ had that all points of degree 5 were not incident to any edges of bottleneck length. As such, the algorithm capitalised on this and was able to produce $4$-MBSTs whose bottleneck values were the same as that of the MST. It is still worth noting that it was the only algorithm amongst those tested that was capable of achieving this for every test instance.\\
Finally, the DNLS algorithm appeared to be a very successful algorithm for $\delta = 3$ and $4$ as it was the best performing algorithm for $n = 100$ in terms of weight and competitive as a $\delta$-MBST algorithm for both these degree bounds. Whilst its time complexity was much greater than that of $\delta$-Prim's and Chan4 and may make it impractical, it does demonstrate a successful proof-of-concept of the diminishing neighbourhood approach and the idea of generalising the recursive edge swap algorithms through similar means. It would certainly be worth investigating in future in order to find a modification of DNLS or another realisation of its polyhedral description that is more time efficient, as the current results look promising.\\
\bibliographystyle{apalike}
\bibliography{ref}

\begin{appendices}

\section{Results Tables for $\delta = 2$}\label{app:delta2}

\begin{table}[H]
\centering
\caption{Average Total Weight: $\delta = 2$}
\label{tab:2MST_sum1}
\begin{tabular}{@{}llllll@{}}
\toprule
$n$ & BCLS     & FLS       & FWLS     & FWLS-B   & DNLS     \\ \midrule
10  & 24436.14 & 26790.42  & 24929.37 & 25949.92 & 23417.45 \\
20  & 37033.01 & 44319.21  & 37408.68 & 39249.56 & 36108.54 \\
30  & 45729.68 & 62351.99  & 46335.58 & 47483.09 & 44192.98 \\
40  & 51361.02 & 73800.72  & 52295.04 & 54035.36 & 51880.75 \\
50  & 59689.96 & 88917.51  & 59535.03 & 63242.02 & 58376.23 \\
60  & 65350.23 & 104693.43 & 65975.54 & 68164.47 & 65693.20 \\
70  & 68861.58 & 114060.09 & 69051.98 & 73277.70 & 68276.69 \\
80  & 75548.85 & 130872.63 & 76203.19 & 80183.17 & 75400.67 \\
90  & 79064.64 & 139950.00 & 79268.85 & 84316.37 & 80527.23 \\
100 & 84149.09 & 152009.40 & 84320.85 & 88952.35 & 84649.89 \\ \bottomrule
\end{tabular}
\end{table}

\begin{table}[H]
\centering
\caption{Average Total Weight: $\delta = 2$}
\label{tab:2MST_sum2}
\begin{tabular}{@{}lllllll@{}}
\toprule
$n$ & 2Prim    & MHC       & DT    & Christofides    & Cube2     & MST      \\ \midrule
10  & 22899.99 & 22005.13  & 24384.58 & 25775.19 & 33007.36  & 20758.99 \\
20  & 35841.01 & 35515.40  & 40359.06 & 41683.00 & 52007.06  & 31280.11 \\
30  & 45178.80 & 47960.45  & 49637.63 & 49013.65 & 62878.63  & 37149.45 \\
40  & 52677.29 & 58660.58  & 56539.27 & 57058.29 & 71673.03  & 42431.98 \\
50  & 60511.64 & 71956.39  & 66408.32 & 63462.61 & 81553.66  & 48470.71 \\
60  & 66201.49 & 82878.24  & 73849.04 & 68919.92 & 89947.48  & 52989.35 \\
70  & 70862.48 & 92089.03  & 78853.47 & 75045.97 & 95024.31  & 55929.58 \\
80  & 78167.97 & 103650.81 & 87170.08 & 78857.75 & 103489.62 & 60923.35 \\
90  & 82387.43 & 110884.87 & 90396.64 & 84440.60 & 109829.20 & 64223.84 \\
100 & 85858.88 & 120537.33 & 96651.10 & 88526.94 & 115281.50 & 67721.19 \\ \bottomrule
\end{tabular}
\end{table}

\begin{table}[H]
\centering
\caption{Average Bottleneck: $\delta = 2$}
\label{tab:2MST_bot1}
\begin{tabular}{@{}llllll@{}}
\toprule
$n$ & BCLS    & FLS      & FWLS    & FWLS-B  & DNLS    \\ \midrule
10  & 6313.72 & 6932.33  & 6296.99 & 6424.16 & 5522.44 \\
20  & 5768.36 & 8170.30  & 5791.13 & 6434.38 & 4901.67 \\
30  & 7067.59 & 9705.99  & 7205.96 & 6502.52 & 5326.91 \\
40  & 6034.60 & 9271.62  & 6412.21 & 6458.69 & 5878.87 \\
50  & 6840.79 & 9850.23  & 6481.29 & 7311.12 & 4681.73 \\
60  & 6964.32 & 9694.59  & 7354.63 & 6242.29 & 5906.97 \\
70  & 6530.97 & 10328.01 & 6554.73 & 6707.74 & 4876.63 \\
80  & 6503.95 & 10530.36 & 6797.12 & 6937.40 & 5485.94 \\
90  & 6400.43 & 10678.68 & 6506.56 & 6685.90 & 6440.13 \\
100 & 6678.62 & 10180.19 & 6626.39 & 7212.88 & 5890.03 \\ \bottomrule
\end{tabular}
\end{table}

\begin{table}[H]
\centering
\caption{Average Bottleneck: $\delta = 2$}
\label{tab:2MST_bot2}
\begin{tabular}{@{}lllllll@{}}
\toprule
$n$ & 2Prim   & MHC     & DT   & Christofides   & Cube2   & MST     \\ \midrule
10  & 5477.91 & 4937.49 & 6081.60 & 6495.97 & 6760.89 & 4415.40 \\
20  & 5171.00 & 4232.63 & 6261.62 & 7496.07 & 5376.43 & 3335.64 \\
30  & 5549.48 & 4391.61 & 6248.40 & 6398.51 & 4643.77 & 2699.20 \\
40  & 5677.86 & 4146.16 & 5548.35 & 6987.72 & 4023.43 & 2397.57 \\
50  & 5456.46 & 4882.08 & 5782.52 & 5992.70 & 3629.19 & 2212.94 \\
60  & 5407.46 & 5073.24 & 5800.91 & 5436.32 & 3629.61 & 2043.17 \\
70  & 5634.05 & 5219.76 & 5519.14 & 6582.66 & 3212.39 & 1917.25 \\
80  & 5466.14 & 5265.07 & 5673.11 & 5327.71 & 3061.51 & 1738.30 \\
90  & 5513.12 & 5107.40 & 5227.09 & 5850.77 & 2935.15 & 1762.46 \\
100 & 5370.05 & 5203.57 & 5188.58 & 5312.55 & 2697.41 & 1619.13 \\ \bottomrule
\end{tabular}
\end{table}

\begin{table}[H]
\centering
\caption{Average Time: $\delta = 2$}
\label{tab:2MST_tim1}
\begin{tabular}{@{}llllll@{}}
\toprule
$n$ & BCLS   & FLS    & FWLS   & FWLS-B & DNLS    \\ \midrule
10  & 0.000  & 0.000  & 0.000  & 0.000  & 0.003   \\
20  & 0.012  & 0.001  & 0.004  & 0.014  & 0.242   \\
30  & 0.088  & 0.027  & 0.042  & 0.090  & 0.846   \\
40  & 0.405  & 0.123  & 0.182  & 0.212  & 3.388   \\
50  & 1.182  & 0.382  & 0.558  & 0.489  & 10.775  \\
60  & 3.801  & 1.088  & 1.556  & 1.326  & 28.983  \\
70  & 6.364  & 2.071  & 2.970  & 2.615  & 61.254  \\
80  & 13.367 & 3.942  & 5.931  & 5.347  & 117.443 \\
90  & 22.320 & 6.646  & 9.915  & 9.403  & 212.995 \\
100 & 37.542 & 11.175 & 16.611 & 15.451 & 359.226 \\ \bottomrule
\end{tabular}
\end{table}

\begin{table}[H]
\centering
\caption{Average Time: $\delta = 2$}
\label{tab:2MST_tim2}
\begin{tabular}{@{}lllllll@{}}
\toprule
$n$ & 2Prim & MHC    & DT & Christofides & Cube2 & MST   \\ \midrule
10  & 0.000 & 1.041  & 0.000 & 0.000 & 0.000 & 0.000 \\
20  & 0.000 & 2.897  & 0.000 & 0.000 & 0.000 & 0.000 \\
30  & 0.000 & 5.257  & 0.000 & 0.000 & 0.000 & 0.000 \\
40  & 0.000 & 8.902  & 0.000 & 0.000 & 0.000 & 0.000 \\
50  & 0.000 & 12.713 & 0.000 & 0.000 & 0.000 & 0.000 \\
60  & 0.000 & 10.359 & 0.000 & 0.000 & 0.000 & 0.000 \\
70  & 0.000 & 13.862 & 0.000 & 0.000 & 0.000 & 0.000 \\
80  & 0.000 & 17.436 & 0.000 & 0.000 & 0.000 & 0.000 \\
90  & 0.000 & 21.555 & 0.000 & 0.000 & 0.000 & 0.000 \\
100 & 0.000 & 25.902 & 0.000 & 0.000 & 0.000 & 0.000 \\ \bottomrule
\end{tabular}
\end{table}

\section{Results Tables for $\delta = 3$ with Uniform Instances}\label{app:delta3U}

\begin{table}[H]
\centering
\caption{Average Total Weight: $\delta = 3$ (Uniform Instances)}
\label{tab:3MSTU_sum1}
\begin{tabular}{@{}llllll@{}}
\toprule
$n$ & BCLS      & FLS      & FWLS      & FWLS-B 	& DNLS        \\ \midrule
10  & 25682.30  & 25682.30 & 25682.30  & 25682.30 & 22916.40     \\
20  & 36764.13 	& 36883.04 & 36705.84  & 36883.04 & 32154.42 \\
30  & 42925.75  & 43154.84 & 43060.04  & 42763.18 & 38010.85    \\
40  & 48101.59  & 48290.78 & 48290.78  & 48290.78 & 44176.54   \\
50  & 53069.75 	& 54560.06 & 53180.35  & 53633.61 & 49131.62 \\
60  & 57465.75 	& 58624.12 & 58242.22  & 58412.87 & 53578.16 \\
70  & 61318.47  & 62768.17 & 61905.92  & 62029.31	& 57154.67  \\
80  & 65311.42  & 67649.79 & 65929.11  & 66059.01 & 61139.75   \\
90  & 68296.37 	& 70011.14 & 68766.19  & 69259.45	& 64548.94 \\
100 & 72081.08 	& 73947.71 & 72722.12  & 73241.11 & 67541.56 \\ \bottomrule
\end{tabular}
\end{table}

\begin{table}[H]
\centering
\caption{Average Total Weight: $\delta = 3$ (Uniform Instances)}
\label{tab:3MSTU_sum2}
\begin{tabular}{@{}lllllll@{}}
\toprule
$n$ & 3Prim       & MHC         & KRY         & KRY-B       & Chan3       & MST         \\ \midrule
10  & 23406.25    & 22829.55    & 25485.40    & 25678.40    & 25152.25    & 22654.20     \\
20  & 31875.92 	  & 32010.36 	& 35491.36    & 35715.53 	& 34954.08 	  & 31729.52 \\
30  & 37477.55    & 38901.44 	& 42294.34    & 42618.85    & 41984.82 	  & 37274.45    \\
40  & 43880.83    & 47081.31   	& 49511.53    & 49849.78   	& 50898.06    & 43734.01    \\
50  & 48678.84    & 53762.36 	& 54871.67    & 55147.91 	& 54346.02    & 48489.95 \\
60  & 53492.13 	  & 60188.51  	& 60403.32    & 60589.79 	& 60633.20 	  & 53392.21 \\
70  & 57017.12    & 65020.34   	& 64618.32    & 64907.08    & 64107.30    & 56872.27    \\
80  & 60832.07    & 70220.48   	& 69499.21    & 69709.75   	& 68643.59    & 60709.59   \\
90  & 64436.57    & 75089.09 	& 73778.04    & 74002.11 	& 74597.86 	  & 64325.80 \\
100 & 67591.66    & 79139.97 	& 77105.00    & 77297.67  	& 75932.13    & 67466.25 \\ \bottomrule
\end{tabular}
\end{table}


\begin{table}[H]
\centering
\caption{Average Bottleneck: $\delta = 3$ (Uniform Instances)}
\label{tab:3MSTU_bot1}
\begin{tabular}{@{}llllll@{}}
\toprule
$n$ & BCLS      & FLS       & FWLS      & FWLS-B 	& DNLS        \\ \midrule
10  & 5625.90   & 5625.90   & 5625.90   & 5625.90   & 4248.89     \\
20  & 6446.53 	& 6485.10   & 6386.30 	& 6485.10   & 3374.81 \\
30  & 6589.03 	& 6687.39 	& 6687.39 	& 6288.41 	& 3252.14 \\
40  & 5303.87   & 5383.15   & 5383.15   & 5383.15 	& 2484.68    \\
50  & 5511.58 	& 5872.09 	& 5558.89   & 5719.66 	& 2619.31 \\
60  & 4932.31 	& 5732.28 	& 5618.82 	& 5732.28 	& 2072.94 \\
70  & 5157.35   & 5914.52  	& 5677.72  	& 5669.40	& 2058.48   \\
80  & 5355.30   & 6184.37   & 5824.56   & 5776.99	& 2073.49   \\
90  & 4594.43 	& 5304.18 	& 5045.34 	& 5244.13 	& 1916.87 \\
100 & 5239.38  	& 6103.43 	& 5805.91  	& 5669.85 	& 1642.86 \\ \bottomrule
\end{tabular}
\end{table}

\begin{table}[H]
\centering
\caption{Average Bottleneck: $\delta = 3$ (Uniform Instances)}
\label{tab:3MSTU_bot2}
\begin{tabular}{@{}lllllll@{}}
\toprule
$n$ & 3Prim     & MHC       & KRY       & KRY-B     & Chan3     & MST         \\ \midrule
10  & 4248.89   & 4248.89   & 5037.21   & 5003.99   & 4930.90   & 4113.68     \\
20  & 3258.73 	& 3274.75 	& 3795.03 	& 3795.03 	& 3757.70 	& 3205.22 \\
30  & 2772.41 	& 2905.20 	& 3341.83 	& 3341.83 	& 3417.23 	& 2772.41 \\
40  & 2328.50   & 2596.53   & 2911.13   & 2893.56   & 3250.26   & 2313.72    \\
50  & 2268.06 	& 2522.68 	& 2616.91 	& 2607.16 	& 2687.46 	& 2257.12    \\
60  & 2042.27 	& 2377.95  	& 2466.83 	& 2466.83 	& 2600.42 	& 2042.27 \\
70  & 1933.13   & 2221.81  	& 2348.53  	& 2348.53  	& 2437.15  	& 1926.63   \\
80  & 1814.28   & 2070.19   & 2270.66   & 2268.80   & 2215.27   & 1807.66    \\
90  & 1813.84 	& 2034.72 	& 2180.44 	& 2180.44 	& 2428.33  	& 1808.27 \\
100 & 1641.07 	& 1953.77 	& 2057.38 	& 2057.38 	& 2074.78 	& 1640.00 \\ \bottomrule
\end{tabular}
\end{table}


\begin{table}[H]
\centering
\caption{Average Time: $\delta = 3$ (Uniform Instances)}
\label{tab:3MSTU_tim1}
\begin{tabular}{@{}llllll@{}}
\toprule
$n$ & BCLS      & FLS     	& FWLS      & FWLS-B    & DNLS        \\ \midrule
10  & 0.000     & 0.000     & 0.000     & 0.000     & 0.000           \\
20  & 0.000     & 0.000     & 0.000     & 0.000     & 0.004 \\
30  & 0.001 	& 0.000     & 0.000     & 0.000     & 0.168 \\
40  & 0.001     & 0.000     & 0.000     & 0.000     & 0.298       \\
50  & 0.018 	& 0.008   	& 0.009 	& 0.008     & 1.337 \\
60  & 0.008   	& 0.003 	& 0.003   	& 0.003   	& 3.441    \\
70  & 0.024     & 0.010 	& 0.011     & 0.010     & 6.763     \\
80  & 0.055     & 0.026  	& 0.030     & 0.026     & 9.045      \\
90  & 0.094 	& 0.045 	& 0.044 	& 0.044  	& 20.913 \\
100 & 0.083 	& 0.039 	& 0.040     & 0.984 	& 35.816 \\ \bottomrule
\end{tabular}
\end{table}

\begin{table}[H]
\centering
\caption{Average Time: $\delta = 3$ (Uniform Instances)}
\label{tab:3MSTU_tim2}
\begin{tabular}{@{}lllllll@{}}
\toprule
$n$ & 3Prim & MHC    & KRY 	   & KRY-B 	   & Chan3 	   & MST \\ \midrule
10  & 0.000 & 1.310  & 0.000   & 0.000     & 0.000     & 0.000   \\
20  & 0.000 & 5.518  & 0.000   & 0.000     & 0.000     & 0.000   \\
30  & 0.000 & 12.878 & 0.000   & 0.000     & 0.000     & 0.000   \\
40  & 0.000 & 24.723 & 0.000   & 0.000     & 0.000     & 0.000   \\
50  & 0.000 & 41.210 & 0.000   & 0.000     & 0.000     & 0.000   \\
60  & 0.000 & 52.502 & 0.000   & 0.000     & 0.000     & 0.000   \\
70  & 0.007 & 72.366 & 0.000   & 0.000     & 0.000     & 0.000   \\
80  & 0.010 & 89.660 & 0.000   & 0.000     & 0.000     & 0.000   \\
90  & 0.010 & 68.229 & 0.000   & 0.000     & 0.000     & 0.000   \\
100 & 0.020 & 83.154 & 0.000   & 0.000     & 0.000     & 0.000   \\ \bottomrule
\end{tabular}
\end{table}

\section{Results Tables for $\delta = 3$ with Special Instances}\label{app:delta3S}

\begin{table}[H]
\centering
\caption{Average Total Weight: $\delta = 3$ (Special Instances)}
\label{tab:3MSTS_sum1}
\begin{tabular}{@{}llllll@{}}
\toprule
$n$ & BCLS     & FLS       & FWLS     & FWLS-B   & DNLS     \\ \midrule
11  & 11510.31 & 17860.40  & 11781.81 & 13829.91 & 8972.46  \\
20  & 23584.23 & 35972.27  & 24011.18 & 29849.14 & 20888.58 \\
30  & 32955.17 & 49833.01  & 33331.48 & 39477.20 & 29319.45 \\
40  & 34228.30 & 65160.06  & 34448.93 & 46194.54 & 30137.58 \\
50  & 39969.19 & 75696.69  & 40428.30 & 51262.56 & 36052.87 \\
60  & 42269.20 & 96610.66  & 42702.05 & 57860.50 & 38282.85 \\
70  & 46791.84 & 108093.78 & 47249.57 & 62514.30 & 42917.58 \\
80  & 47844.30 & 117776.01 & 48316.64 & 68296.15 & 44698.47 \\
90  & 53244.97 & 132551.63 & 54024.13 & 72494.67 & 48923.77 \\
100 & 54174.67 & 146349.57 & 54146.47 & 75934.45 & 49379.68 \\ \bottomrule
\end{tabular}
\end{table}

\begin{table}[H]
\centering
\caption{Average Total Weight: $\delta = 3$ (Special Instances)}
\label{tab:3MSTS_sum2}
\begin{tabular}{@{}lllllll@{}}
\toprule
$n$ & 3Prim    & MHC      & KRY      & KRY-B    & Chan3    & MST      \\ \midrule
11  & 8679.31  & 8708.96  & 8988.29  & 9153.63  & 8988.29  & 8588.93  \\
20  & 20750.36 & 20705.76 & 22369.10 & 22725.05 & 22363.75 & 20631.59 \\
30  & 29245.25 & 29453.70 & 31612.39 & 31965.53 & 31583.69 & 29033.00 \\
40  & 30192.49 & 30812.91 & 32697.70 & 33074.88 & 32656.64 & 30040.05 \\
50  & 36105.30 & 37681.84 & 38778.24 & 39202.98 & 38807.61 & 35929.10 \\
60  & 38211.63 & 40230.86 & 41352.82 & 41953.43 & 41132.43 & 38001.00 \\
70  & 42919.69 & 46408.62 & 46602.36 & 47249.00 & 46508.06 & 42733.89 \\
80  & 44538.42 & 48104.88 & 48311.90 & 48857.11 & 48108.64 & 44315.79 \\
90  & 49087.91 & 53774.59 & 53178.58 & 53838.66 & 52885.92 & 48803.36 \\
100 & 49534.09 & 55379.59 & 53986.24 & 55132.13 & 53869.64 & 49247.62 \\ \bottomrule
\end{tabular}
\end{table}


\begin{table}[H]
\centering
\caption{Average Bottleneck: $\delta = 3$ (Special Instances)}
\label{tab:3MSTS_bot1}
\begin{tabular}{@{}llllll@{}}
\toprule
$n$ & BCLS    & FLS     & FWLS    & FWLS-B  & DNLS    \\ \midrule
11  & 4428.74 & 5434.09 & 4738.48 & 5209.04 & 3908.82 \\
20  & 5238.53 & 6760.84 & 5444.13 & 6250.26 & 4310.00 \\
30  & 5214.08 & 7221.60 & 5494.63 & 6076.14 & 3504.66 \\
40  & 5062.99 & 7617.41 & 5227.41 & 6350.87 & 3445.02 \\
50  & 4877.66 & 8014.15 & 5240.71 & 5314.81 & 3112.57 \\
60  & 5163.50 & 8962.21 & 5569.41 & 5949.04 & 3049.73 \\
70  & 4712.58 & 8929.46 & 5136.77 & 6080.47 & 2634.34 \\
80  & 4336.74 & 8356.78 & 4794.80 & 5927.70 & 2693.28 \\
90  & 4990.15 & 9023.48 & 5765.57 & 6146.28 & 2420.19 \\
100 & 4978.01 & 9152.09 & 5309.02 & 6746.44 & 2254.82 \\ \bottomrule
\end{tabular}
\end{table}

\begin{table}[H]
\centering
\caption{Average Bottleneck: $\delta = 3$ (Special Instances)}
\label{tab:3MSTS_bot2}
\begin{tabular}{@{}lllllll@{}}
\toprule
$n$ & 3Prim   & MHC     & KRY     & KRY-B   & Chan3   & MST     \\ \midrule
11  & 3897.00 & 3923.95 & 3897.00 & 3897.00 & 3897.00 & 3896.33 \\
20  & 4310.00 & 4310.00 & 4649.04 & 4649.04 & 4650.01 & 4310.00 \\
30  & 3505.25 & 3515.05 & 4004.21 & 4004.21 & 4013.85 & 3504.55 \\
40  & 3444.94 & 3469.05 & 3760.02 & 3760.02 & 3727.68 & 3444.94 \\
50  & 3112.30 & 3176.78 & 3392.26 & 3392.26 & 3414.17 & 3112.30 \\
60  & 3049.58 & 3122.99 & 3343.18 & 3343.18 & 3314.07 & 3049.58 \\
70  & 2641.88 & 2822.75 & 3108.04 & 3108.04 & 3133.04 & 2634.27 \\
80  & 2578.76 & 2694.35 & 2867.22 & 2867.22 & 2904.84 & 2578.76 \\
90  & 2421.45 & 2534.66 & 2757.33 & 2757.33 & 2712.04 & 2420.19 \\
100 & 2252.99 & 2403.81 & 2500.99 & 2469.12 & 2557.82 & 2252.99 \\ \bottomrule
\end{tabular}
\end{table}


\begin{table}[H]
\centering
\caption{Average Time: $\delta = 3$ (Special Instances)}
\label{tab:3MSTS_tim1}
\begin{tabular}{@{}llllll@{}}
\toprule
$n$ & BCLS   & FLS   & FWLS  & FWLS-B & DNLS    \\ \midrule
11  & 0.000  & 0.000 & 0.000 & 0.000  & 0.001   \\
20  & 0.015  & 0.000 & 0.000 & 0.000  & 0.066   \\
30  & 0.051  & 0.023 & 0.032 & 0.018  & 0.308   \\
40  & 0.299  & 0.115 & 0.126 & 0.090  & 1.532   \\
50  & 0.724  & 0.306 & 0.359 & 0.255  & 4.841   \\
60  & 1.776  & 0.597 & 0.814 & 0.563  & 11.840  \\
70  & 3.232  & 1.189 & 1.596 & 1.129  & 24.627  \\
80  & 7.302  & 2.472 & 3.652 & 2.425  & 50.050  \\
90  & 10.835 & 4.066 & 5.345 & 4.111  & 91.916  \\
100 & 19.177 & 6.625 & 9.396 & 6.629  & 157.411 \\ \bottomrule
\end{tabular}
\end{table}

\begin{table}[H]
\centering
\caption{Average Time: $\delta = 3$ (Special Instances)}
\label{tab:3MSTS_tim2}
\begin{tabular}{@{}lllllll@{}}
\toprule
$n$ & 3Prim & MHC    & KRY   & KRY-B & Chan3 & MST   \\ \midrule
11  & 0.000 & 0.668  & 0.000 & 0.000 & 0.000 & 0.000 \\
20  & 0.000 & 2.211  & 0.000 & 0.000 & 0.000 & 0.000 \\
30  & 0.000 & 5.172  & 0.000 & 0.000 & 0.000 & 0.000 \\
40  & 0.000 & 9.974  & 0.000 & 0.000 & 0.000 & 0.000 \\
50  & 0.001 & 16.788 & 0.000 & 0.000 & 0.000 & 0.000 \\
60  & 0.002 & 25.701 & 0.000 & 0.000 & 0.000 & 0.000 \\
70  & 0.005 & 45.562 & 0.000 & 0.000 & 0.000 & 0.000 \\
80  & 0.012 & 64.226 & 0.000 & 0.000 & 0.000 & 0.000 \\
90  & 0.016 & 67.948 & 0.000 & 0.000 & 0.000 & 0.000 \\
100 & 0.027 & 88.843 & 0.000 & 0.000 & 0.000 & 0.000 \\ \bottomrule
\end{tabular}
\end{table}

\section{Results Tables for $\delta = 4$}\label{app:delta4}

\begin{table}[H]
\centering
\caption{Average Total Weight: $\delta = 4$}
\label{tab:4MST_sum1}
\begin{tabular}{@{}llllll@{}}
\toprule
$n$ & BCLS     & FLS      & FWLS     & FWLS-B   & DNLS     \\ \midrule
11  & 12619.69 & 13189.22 & 13067.56 & 13189.22 & 8787.21  \\
20  & 25391.85 & 25765.28 & 25765.28 & 25765.28 & 20659.18 \\
30  & 33721.79 & 34565.89 & 34325.23 & 34489.13 & 29047.30 \\
40  & 34789.61 & 41059.53 & 35116.89 & 37411.40 & 30050.51 \\
50  & 40572.80 & 46188.30 & 40947.49 & 42410.97 & 35937.90 \\
60  & 43148.75 & 54768.67 & 43764.76 & 46862.51 & 38010.43 \\
70  & 47216.39 & 58576.44 & 47839.22 & 51273.96 & 42742.83 \\
80  & 48353.99 & 65576.80 & 49399.79 & 54940.46 & 44324.53 \\
90  & 52840.98 & 69382.32 & 53293.05 & 58280.61 & 48811.68 \\
100 & 54258.99 & 77307.82 & 54606.77 & 61095.96 & 49258.77 \\ \bottomrule
\end{tabular}
\end{table}

\begin{table}[H]
\centering
\caption{Average Total Weight: $\delta = 4$}
\label{tab:4MST_sum2}
\begin{tabular}{@{}lllll@{}}
\toprule
$n$ & 4Prim    & MHC      & Chan4    & MST      \\ \midrule
11  & 8617.45  & 8593.28  & 8603.74  & 8588.93  \\
20  & 20666.98 & 20642.18 & 20649.14 & 20631.59 \\
30  & 29065.72 & 29254.03 & 29050.52 & 29033.00 \\
40  & 30074.96 & 30490.85 & 30059.04 & 30040.05 \\
50  & 35970.49 & 37137.77 & 35950.42 & 35929.10 \\
60  & 38031.81 & 39639.00 & 38022.48 & 38001.00 \\
70  & 42766.86 & 45169.44 & 42765.87 & 42733.89 \\
80  & 44362.79 & 46863.59 & 44344.67 & 44315.79 \\
90  & 48842.45 & 52288.03 & 48825.86 & 48803.36 \\
100 & 49299.02 & 52703.00 & 49273.87 & 49247.62 \\ \bottomrule
\end{tabular}
\end{table}

\begin{table}[H]
\centering
\caption{Average Bottleneck: $\delta = 4$}
\label{tab:4MST_bot1}
\begin{tabular}{@{}llllll@{}}
\toprule
$n$ & BCLS    & FLS     & FWLS    & FWLS-B  & DNLS    \\ \midrule
11  & 4596.08 & 5099.73 & 5080.89 & 5099.73 & 3908.73 \\
20  & 5735.74 & 5913.40 & 5913.40 & 5913.40 & 4310.00 \\
30  & 5336.06 & 5922.63 & 5840.33 & 5922.63 & 3505.05 \\
40  & 5251.60 & 7161.23 & 5554.84 & 5977.89 & 3445.00 \\
50  & 5215.91 & 6282.23 & 5533.39 & 5154.56 & 3112.34 \\
60  & 5607.06 & 7557.34 & 6143.23 & 6059.88 & 3049.69 \\
70  & 4967.64 & 6948.15 & 5527.75 & 5909.07 & 2634.30 \\
80  & 4548.09 & 7797.90 & 5418.94 & 6047.61 & 2578.83 \\
90  & 4582.64 & 7715.16 & 4938.34 & 4866.36 & 2420.19 \\
100 & 5309.50 & 8321.79 & 5630.97 & 6654.88 & 2253.15 \\ \bottomrule
\end{tabular}
\end{table}

\begin{table}[H]
\centering
\caption{Average Bottleneck: $\delta = 4$}
\label{tab:4MST_bot2}
\begin{tabular}{@{}lllll@{}}
\toprule
$n$ & 4Prim   & MHC     & Chan4   & MST     \\ \midrule
11  & 3897.00 & 3896.33 & 3896.33 & 3896.33 \\
20  & 4310.00 & 4310.00 & 4310.00 & 4310.00 \\
30  & 3504.55 & 3519.86 & 3504.55 & 3504.55 \\
40  & 3444.94 & 3463.23 & 3444.94 & 3444.94 \\
50  & 3112.30 & 3140.28 & 3112.30 & 3112.30 \\
60  & 3049.58 & 3101.19 & 3049.58 & 3049.58 \\
70  & 2634.27 & 2692.91 & 2634.27 & 2634.27 \\
80  & 2578.76 & 2657.34 & 2578.76 & 2578.76 \\
90  & 2420.19 & 2512.55 & 2420.19 & 2420.19 \\
100 & 2252.99 & 2329.67 & 2252.99 & 2252.99 \\ \bottomrule
\end{tabular}
\end{table}

\begin{table}[H]
\centering
\caption{Average Time: $\delta = 4$}
\label{tab:4MST_tim1}
\begin{tabular}{@{}llllll@{}}
\toprule
$n$ & BCLS  & FLS   & FWLS  & FWLS-B & DNLS   \\ \midrule
11  & 0.000 & 0.000 & 0.000 & 0.000  & 0.000  \\
20  & 0.000 & 0.000 & 0.000 & 0.000  & 0.009  \\
30  & 0.000 & 0.000 & 0.000 & 0.000  & 0.151  \\
40  & 0.014 & 0.006 & 0.007 & 0.007  & 0.736  \\
50  & 0.040 & 0.017 & 0.016 & 0.017  & 2.201  \\
60  & 0.124 & 0.056 & 0.056 & 0.055  & 5.881  \\
70  & 0.200 & 0.093 & 0.095 & 0.091  & 12.639 \\
80  & 0.599 & 0.312 & 0.295 & 0.273  & 23.798 \\
90  & 0.800 & 0.401 & 0.393 & 0.385  & 43.311 \\
100 & 1.317 & 0.614 & 0.646 & 0.589  & 81.804 \\ \bottomrule
\end{tabular}
\end{table}

\begin{table}[H]
\centering
\caption{Average Time: $\delta = 4$}
\label{tab:4MST_tim2}
\begin{tabular}{@{}lllll@{}}
\toprule
$n$ & 4Prim & MHC     & Chan4 & MST   \\ \midrule
11  & 0.000 & 1.052   & 0.000 & 0.000 \\
20  & 0.000 & 3.717   & 0.000 & 0.000 \\
30  & 0.000 & 8.990   & 0.000 & 0.000 \\
40  & 0.000 & 17.613  & 0.000 & 0.000 \\
50  & 0.000 & 30.032  & 0.000 & 0.000 \\
60  & 0.000 & 41.364  & 0.000 & 0.000 \\
70  & 0.008 & 59.632  & 0.000 & 0.000 \\
80  & 0.010 & 68.302  & 0.000 & 0.000 \\
90  & 0.010 & 84.047  & 0.000 & 0.000 \\
100 & 0.020 & 106.873 & 0.000 & 0.000 \\ \bottomrule
\end{tabular}
\end{table}

\end{appendices}

\end{document}